\documentclass[10pt,a4paper]{article}
\usepackage[latin9]{inputenc}
\usepackage{mathtools}
\usepackage{amsmath}
\usepackage{amsthm}
\usepackage{amssymb}
\usepackage{esint}

\makeatletter

\pdfpageheight\paperheight
\pdfpagewidth\paperwidth

 \theoremstyle{definition}
 \newtheorem*{defn*}{\protect\definitionname}
\theoremstyle{plain}
\newtheorem{thm}{\protect\theoremname}[section]
  \theoremstyle{plain}
  \newtheorem{prop}[thm]{\protect\propositionname}
  \theoremstyle{definition}
  \newtheorem{example}[thm]{\protect\examplename}
  \theoremstyle{plain}
  \newtheorem{lem}[thm]{\protect\lemmaname}
  \theoremstyle{plain}
  \newtheorem{cor}[thm]{\protect\corollaryname}
  \theoremstyle{remark}
  \newtheorem{rem}[thm]{\protect\remarkname}

\usepackage{amsthm}\usepackage{multirow}\usepackage{graphicx}\usepackage{ifpdf}
\usepackage{enumerate}
\usepackage{url}

\numberwithin{equation}{section}
\usepackage{enumitem}
\setlist[enumerate]{label*=(\alph*),ref=(\alph*)}

\usepackage{authblk}

\allowdisplaybreaks

\title{A Hilbert space approach to\\ fractional differential equations}

\author[1]{Kai Diethelm}
\author[2]{Konrad Kitzing}
\author[2]{Rainer Picard}
\author[2]{Stefan Siegmund}
\author[3]{Sascha Trostorff}
\author[4]{Marcus Waurick}
\affil[1]{Fakultät Angewandte Natur- und Geisteswissenschaften, Hochschule für angewandte Wissenschaften Würzburg-Schweinfurt, Germany}
\affil[2]{Institute of Analysis, Faculty of Mathematics, TU Dresden, Germany}
\affil[3]{Mathematisches Seminar, Christian-Albrechts-Universität zu Kiel, Germany}
\affil[4]{Department of Mathematics and Statistics, University of Strathclyde, Glasgow,
Scotland}
\date{}                     
\newcommand{\C}{\mathbb{C}}
\newcommand{\R}{\mathbb{R}}
\newcommand{\N}{\mathbb{N}}

\DeclareMathOperator{\dom}{dom}

\newcommand{\hide}[1]{}

\newcommand*{\e}{\mathrm{e}}
\renewcommand*{\i}{\mathrm{i}}

\DeclareMathAccent{\Circ}{\mathalpha}{operators}{"17}

\renewcommand{\Re}{\operatorname{\mathrm{Re}}}

\newcommand{\ii}{\mathrm {i}}
\newcommand{\spt}{\operatorname{spt}}

\newcommand{\norm}[1]{\left\lVert#1\right\rVert}
\newcommand{\abs}[1]{\left\lvert#1\right\rvert}

\newcommand{\cC}{C_{\mathrm{c}}}
\renewcommand{\d}{\; \mathrm{d}}

\renewcommand{\tilde}{\widetilde}
\renewcommand*{\epsilon}{\varepsilon}
\renewcommand*{\theta}{\vartheta}
\renewcommand*{\rho}{\varrho}

\makeatother

  \providecommand{\corollaryname}{Corollary}
  \providecommand{\definitionname}{Definition}
  \providecommand{\examplename}{Example}
  \providecommand{\lemmaname}{Lemma}
  \providecommand{\propositionname}{Proposition}
  \providecommand{\remarkname}{Remark}
\providecommand{\theoremname}{Theorem}

\usepackage{xcolor}

\DeclareMathOperator*{\esssup}{ess\thinspace sup}

\begin{document}
\maketitle

\begin{abstract}
We study fractional differential equations of Riemann-Liouville and Caputo type in Hilbert spaces. Using exponentially weighted spaces of functions defined on $\mathbb{R}$, we define fractional operators by means of a functional calculus using the Fourier transform. Main tools are extrapolation- and interpolation spaces. Main results are the existence and uniqueness of solutions and the causality of solution operators for non-linear fractional differential equations.
\end{abstract}



\section{Introduction}

The concept of a fractional derivative $\partial_{0}^{\alpha}$,
$\alpha \in\, ]0,1]$, which we utilize, will be based on inverting a suitable
continuous extension of the Riemann-Liouville fractional integral
of continuous functions $f\in \cC(\R)$ with compact support given by
\[
t\mapsto\tfrac{1}{\sqrt{2\pi}}\int_{-\infty}^{t}\tfrac{1}{\Gamma(\alpha)}(t-s)^{\alpha-1}f(s)\d s
\]
as an apparently natural interpolation suggested by the iterated kernel
formula for repeated integration. The choice of the lower limit as
$-\infty$ is determined by our wish to study dynamical processes,
for which causality\footnote{\label{fn:Riemann-Liouville-a }Other frequent choices such as
\[
t\mapsto\tfrac{1}{\sqrt{2\pi}}\chi_{_{]a,\infty[}}\left(t\right)\int_{a}^{t}\tfrac{1}{\Gamma(\alpha)}(t-s)^{\alpha-1}f(s) \d s
\]
for $a\in\mathbb{R}$, would lose time-shift invariance (a suggestive
choice is $a=0$), which we consider undesirable. For our choice of
the limit case $a=-\infty$ it should be noted that the Riemann-Liouville
and the Caputo fractional derivative essentially coincide.} should play an important role.
It is a pleasant fact that the classical definition of $\partial_0^\alpha$ in the sense of \cite{Diethelm2003}
coincides with the other natural choice of $\partial_{0}^{\alpha}$
as a function of $\partial_{0}$ in the sense of a spectral function
calculus of a realization of $\partial_{0}$ as a normal operator
in a suitable Hilbert space setting.
This is specified below.
The Hilbert space framework is based on observations in \cite{0753.44002} and has already been exploited for linear fractional partial differential equations in \cite{picard2015evolutionary}.
In this paper, however, we study fractional differentiation and different notions of nonlinear fractional differential equations, using extrapolated fractional Sobolev spaces.
The approach taken in this paper contrasts with other approaches in fractional calculus.
Indeed, in \cite{Zacher2009} a fractional derivative is defined as a derivative of a fractional integral,
in \cite{Gorenflo_und_Co2015} the fractional derivative of $\C$\textendash valued functions on a \emph{bounded} interval and linear fractional differential equations are also studied with a functional calculus and fractional Sobolev spaces.
Here, using the above mentioned functional calculus of the derivative operator, we obtain a causal implementation of the fractional derivative.
The property of causality is not shared by the fractional derivative operator constructed in \cite{Gorenflo_und_Co2015} (cf.\ \cite[Formula (2.3)]{Gorenflo_und_Co2015}).

\section{Fractional derivative in a Hilbert space setting}\label{sec:rho_calc}

In the present section, we introduce the necessary operators to be
used in the following. We will formulate all results in the vector-valued,
more specifically, in the Hilbert space-valued situation.
On a first read, one may think of scalar-valued functions.

To begin with, we introduce an $L^{2}$-variant of the exponentially
weighted space of continuous functions that proved useful in the proof
of the Picard\textendash Lindel{\"o}f Theorem and is attributed to
Morgenstern, \cite{Ref195229}.

We denote by $L^p(\R; H)$ and $L_\mathrm{loc}^1(\R;H)$ the space of $p$-Bochner integrable functions and the space of locally Bochner integrable functions
on a Hilbert space $H$, respectively.
\begin{defn*}
Let $H$ be a Hilbert space, $\rho \in \R$ and $p \in [1, \infty]$.
For $f\in L_\mathrm{loc}^1(\R;H)$ we denote $\e^{-\rho m}f \coloneqq (\R\ni t\mapsto \e^{-\rho t}f(t))$.
We define the normed spaces
\[
L_\rho^p(\R; H) \coloneqq \left\{f\in L_\mathrm{loc}^1(\R;H); \e^{-\rho m} f \in L^p(\R; H)\right\},
\]
with norm
\begin{align*}
        \norm{f}_{L_\rho^p(\R; H)} &\coloneqq \norm{\e^{-\rho m} f}_{L^p(\R; H)} = \left(\int_\R \norm{f(t)}_H^p \e^{-p \rho t} \d t \right)^{1/p} &&(p<\infty),
        \\
        \norm{f}_{L_\rho^p(\R;H)} &\coloneqq \norm{\e^{-\rho m} f}_{L^p(\R;H)} = \esssup \norm{\e^{-\rho m}f}_H &&(p=\infty).
\end{align*}
\end{defn*}

\begin{rem}
The operator $\e^{-\rho m}:L_\rho^p(\R;H)\to L^p(\R;H), f\mapsto \e^{-\rho m}f$ is an isometric isomorphism from $L_\rho^p(\R;H)$ to $L^p(\R;H)$.
Moreover $L_\rho^2(\R; H)$ is a Hilbert space with scalar product
\begin{equation*}
        (f,g)\mapsto\left< f,g\right>_{L_\rho^2(\R; H)} = \int_{\mathbb{R}}\langle f(t),g(t)\rangle_{H}\e^{-2\rho t} \d t.
\end{equation*}
\end{rem}

Next, we introduce the time derivative.

\begin{defn*}
Let $H$ be a Hilbert space.
\begin{enumerate}
\item Let $f,g\in L_{\textnormal{loc}}^{1}(\mathbb{R};H)$. We say that
$f'=g$, if for all $\phi\in C_{c}^{\infty}(\mathbb{R})$
\[
-\int_{\mathbb{R}}f\phi'=\int_{\mathbb{R}}g\phi.
\]
\item Let $\rho\in\mathbb{R}$. We define
\begin{align*}
\partial_{0,\rho} \colon H_{\rho}^1(\mathbb{R};H)\subseteq L_{\rho}^{2}(\mathbb{R};H) &\to L_{\rho}^{2}(\mathbb{R};H)
\\
f &\mapsto f',
\end{align*}
where $H_{\rho}^1(\mathbb{R};H)\coloneqq\{f\in L_{\rho}^{2}(\mathbb{R};H);f'\in L_{\rho}^{2}(\mathbb{R};H)\}$.
\end{enumerate}
\end{defn*}

The index 0 in $\partial_{0,\rho}$ shall indicate that the derivative is with respect to time.
We will introduce the fractional derivatives and fractional integrals
by means of a functional calculus for $\partial_{0,\rho}$. For this, we
introduce the Fourier\textendash Laplace transform.

\begin{defn*}
Let $H$ be a complex Hilbert space.
Let $\rho\in\mathbb{R}$.
\begin{enumerate}
\item We define the Fourier transform of $f\in L^1(\R;H)$ by
\begin{equation*}
        \mathcal F f(\xi) = \frac{1}{\sqrt{2\pi}}\int_{\mathbb{R}}f(t)e^{-\ii\xi t}\,\mathrm{d}t,\qquad \xi\in\R.
\end{equation*}
\item We define the Fourier\textendash Laplace transform on $L_\rho^1(\R;H)$ by $\mathcal{L}_{\rho}\coloneqq\mathcal{F}\e^{-\rho m}$.
\item We define the Fourier transform on $L^2(\R; H)$ denoted $\mathcal{F}\colon L^{2}(\mathbb{R};H)\to L^{2}(\mathbb{R};H)$
        to be the unitary extension of the operator $\mathcal F: L^{1}(\mathbb{R};H)\cap L^{2}(\mathbb{R};H) \to L^{2}(\mathbb{R};H)$.
\item We define the Fourier\textendash Laplace transform on $L_\rho^2(\R; H)$ as the unitary mapping
        $\mathcal{L}_{\rho}\coloneqq\mathcal{F}\e^{-\rho m}:L_\rho^2(\R;H)\to L^2(\R;H)$
\end{enumerate}
\end{defn*}

From now on, $H$ denotes a complex Hilbert space.
With the latter notion at hand, we provide the spectral representation
of $\partial_{0,\rho}$ as the multiplication-by-argument operator
\begin{align*}
        &\dom(m)\coloneqq\{f\in L^2(\R;H); (\R\ni\xi\mapsto\xi f(\xi))\in L^2(\R;H)\},
        \\
        &m:L^2(\R;H)\supseteq\dom(m)\to L^2(\R;H),\qquad f\mapsto (\R\ni\xi\mapsto \xi f(\xi)).
\end{align*}
\begin{thm}
\label{t:spt} Let $\rho\in\mathbb{R}.$ Then
\begin{enumerate}
\item \label{rho0} $\partial_{0,0}=\mathcal{F}^{*}\ii m\mathcal{F}$,
\item \label{rho0.5}$(\e^{-\rho m})^{*}\partial_{0,0}\e^{-\rho m}=\partial_{0,\rho}-\rho$,
\item \label{rho1} $\partial_{0,\rho}=\mathcal{L}_{\rho}^{*}\big(im+\rho\big)\mathcal{L}_{\rho}$.
\end{enumerate}
\end{thm}

\begin{proof}
For the proof of \ref{rho0}, we observe that the equality holds on
the Schwartz space $\mathcal{S}(\mathbb{R};H)$ of smooth, rapidly
decaying functions. In fact, this is an easy application of integration
by parts. The result thus follows from using that $\mathcal{F}$
is a bijection on $\mathcal{S}(\mathbb{R};H)$ and that $\mathcal{S}(\mathbb{R};H)$
is an operator core, for both $m$ and $\partial_{0,0}$.

For the statement \ref{rho0.5} let $f\in H_\rho^1(\R;H)$ and $\varphi\in C_c^\infty(\R)$.
Then $\e^{-\rho m}\varphi\in C_c^\infty(\R)$ and
\begin{align*}
        \int_\R (\e^{-\rho m}f)\varphi' &= \int_\R f (\e^{-\rho m}\varphi')
        \\
        &= \int_\R f \big((\e^{-\rho m}\varphi)'+\rho\e^{-\rho m}\varphi\big)
        \\
        &=-\int_\R \partial_{0,\rho}f\e^{-\rho m}\varphi + \int_\R f \rho\e^{-\rho m}\varphi
        \\
        &= -\int_\R (\e^{-\rho m}\partial_{0,\rho}f-\rho\e^{-\rho m} f)\varphi.
\end{align*}
Hence $\e^{-\rho m}f\in H_0^1(\R;H)$ and $\partial_{0,0}\e^{-\rho m} f = \e^{-\rho m}\partial_{0,\rho}f-\rho\e^{-\rho m} f$.

Next, we address \ref{rho1}. By part \ref{rho0} and \ref{rho0.5},
we compute
\begin{align*}
\partial_{0,\rho} & =(\e^{-\rho m})^{*}\partial_{0,0}\e^{-\rho m}+\rho\\
 & =(\e^{-\rho m})^{*}\mathcal{F}^{*}\ii m\mathcal{F}\e^{-\rho m}+\rho\\
 & =(\e^{-\rho m})^{*}\mathcal{F}^{*}\ii m\mathcal{F}\e^{-\rho m}+(\e^{-\rho m})^{*}\mathcal{F}^{*}\rho\mathcal{F}\e^{-\rho m}\\
 & =\mathcal{L}_{\rho}^{*}\big(\ii m+\rho\big)\mathcal{L}_{\rho}.\qedhere
\end{align*}
\end{proof}

Theorem \ref{t:spt} tells us that $\partial_{0,\rho}$ is unitarily
equivalent to a multiplication operator with spectrum equal to $\i\R+\rho=\{z\in\C;\Re z=\rho\}$.
In particular, we are now in the position
to define \emph{functions} of $\partial_{0,\rho}$.
\begin{defn*}
Let $\rho\in \R$ and  $F\colon \dom(F)\subseteq\{\i t+\rho\, ; \,t\in \R\}\to\mathbb{C}$ be measurable such that $\{t\in \R\,;\,\i t+\rho \notin \dom(F)\}$ has Lebesgue measure zero.
We define
\[
F(\partial_{0,\rho})\coloneqq\mathcal{L}_{\rho}^{*}F\big(\i m+\rho\big)\mathcal{L}_{\rho},
\]
where
\[
F\big(\i m+\rho\big)f\coloneqq\Big(\mathbb{R}\ni\xi\mapsto F\big(\i\xi+\rho\big)f(\xi)\Big)
\]
in case $f\in L^{2}(\mathbb{R};H)$ is such that $\left(\xi\mapsto F\big(\i\xi+\rho\big)f(\xi)\right)\in L^{2}(\mathbb{R};H)$.
\end{defn*}

We record an elementary fact on multiplication operators.
\begin{prop}
\label{p:bdd} Let $F$ be as in the previous definition.
We denote $\norm{F}_{\rho,\infty}\coloneqq\esssup_{\xi\in\R}|F(\i\xi+\rho)|\in[0,\infty]$.
The operator $F(\partial_{0,\rho})$ is bounded, if and only if $\|F\|_{\rho,\infty}<\infty$.
If $F(\partial_{0,\rho})$ is bounded, then $\|F(\partial_{0,\rho})\|=\|F\|_{\rho,\infty}$.
\end{prop}

\begin{proof}
Since $\mathcal L_\rho$ is unitary we may prove that $F(\i m+\rho)$ is bounded on $L^2(\R;H)$ if and only if $\norm{F}_{\rho,\infty}<\infty$.
Suppose $\norm{F}_{\rho,\infty}<\infty$.
Then for $f\in L^2(\R;H)$ we have $\int_\R \norm{F(\i \xi+\rho)f(\xi)}_H^2\d\xi \leq \norm{F}_{\rho,\infty}^2\norm{f}_{L^2}^2$.
Hence $F(\i m+\rho)$ is bounded with $\norm{F(\i m+\rho)}\leq\norm{F}_{\rho,\infty}$.
Let $F(\i m+\rho)$ be bounded.
For $f\in L^2(\R;H)$ with $\norm{f}_{L^2}=1$ we have
\[
\infty>\norm{F(\i m+\rho)}^2\geq \int_\R |F(\i\xi+\rho)|^2\norm{f(\xi)}_H^2 \d\xi.
\]
There is a sequence $(f_n)_{n\in\N}\in L^2(\R;H)^\N$ with $\norm{f_n}_{L^2}=1$ ($n\in\N$)
and $\int_\R |F(\i\xi+\rho)|^2\norm{f_n(\xi)}_H^2 \d\xi\to\esssup|F(\i\cdot+\rho)|^2$ for $n\to\infty$.
This shows $\infty>\norm{F(\i m+\rho)}\geq \norm{F}_{\rho,\infty}$.
To construct $(f_n)_{n\in\N}$ wlog. we may assume that $\esssup|F(\i\cdot+\rho)|>0$.
Let $x\in H$ with $\norm{x}_H=1$ and let $(c_n)_{n\in\N}\in\R^\N$ be a positive sequence with $c_n\uparrow\esssup|F(\i\cdot+\rho)|$.
Then for $n\in\N$ set $A_n\coloneqq [-n,n]\cap\{|F(\i\cdot+\rho)|>c_n\}$.
By the definition of the essential supremum, we may assume that $\lambda(A_n)>0$ and for $n\in\N$ we set
\[
f_n\coloneqq 1_{A_n}\frac{x}{\sqrt{\lambda(A_n)}}.\qedhere
\]
\end{proof}

One important class of operators that can be rooted to be of the form
just introduced are fractional derivatives and fractional integrals:
\begin{example}
\label{exa:deffrac}Let $\alpha \in \mathbb{R}_{>0}$ and $\rho\in \R$. Then the fractional
derivative of order $\alpha$ is given by
\[
\partial_{0,\rho}^{\alpha}=\mathcal{L}_{\rho}^{*}\big(\i m+\rho\big)^{\alpha}\mathcal{L}_{\rho}
\]
and the fractional integral of order $\alpha$ is given by
\[
\partial_{0,\rho}^{-\alpha}=\mathcal{L}_{\rho}^{*}\big(\tfrac{1}{\i m+\rho}\big)^{\alpha}\mathcal{L}_{\rho}.
\]
Note that both expressions are well-defined in the sense of functions of $\partial_{0,\rho}$ defined above and that $\partial_{0,\rho}^{-\alpha}$ is bounded iff $\rho\ne0$. Moreover, $\big(\partial_{0,\rho}^{\alpha}\big)^{-1}=\partial_{0,\rho}^{-\alpha}$.
We set $\partial_{0,\rho}^0$ as the identity operator on $L_\rho^2(\R;H)$.
\end{example}

In order to provide the connections to the more commonly known integral
representation formulas for the fractional integrals, we recall the
multiplication theorem, that is,
\[
\sqrt{2\pi}\mathcal{F}f\cdot\mathcal{F}g=\mathcal{F}(f\ast g),
\]
for $f \in L^1(\R)$ and $g \in L^2(\R; H)$.

We recall the cut-off function
\[
  \chi_{\mathbb{R}_{>0}}(t)\coloneqq\begin{cases}
1, & t > 0,\\
0, & t \leq 0.
\end{cases}
\]

\begin{lem}\label{lem:FL-trafo-frac}
\label{l:latroafo} For all $\rho,\alpha>0$, and $\xi\in\mathbb{R}$,
we have
\begin{equation}\label{eq:laplace}
\sqrt{2\pi} \mathcal{L}_{\rho}\Big(t\mapsto\tfrac{1}{\Gamma(\alpha)}t^{\alpha-1}\chi_{\mathbb{R}_{>0}}(t)\Big)(\xi)=\big(\tfrac{1}{i\xi+\rho}\big)^{\alpha}.
\end{equation}
\end{lem}

\begin{proof}
 We start by defining the function
 \[
  f(\xi)\coloneqq \intop_0^\infty \e^{-(\i \xi +\rho)s} s^{\alpha-1} \d s
 \]
for $\xi\in \R$. Then we have
\begin{align*}
 f'(\xi)&=\intop_0^\infty -\i \e^{-(\i \xi+\rho)s} s^\alpha \d s \\
        &= -\i \frac{\alpha}{\i \xi +\rho} f(\xi),
\end{align*}
where we have used integration by parts. By separation of variables, it follows that
\[
 f(\xi)=f(0) \frac{\rho^\alpha}{(\i \xi+\rho)^\alpha}
\]
for $\xi \in \R$. Now, since
\[
 f(0)=\intop_0^\infty \e^{-\rho s} s^{\alpha-1}\d s= \frac{1}{\rho^\alpha} \Gamma(\alpha),
\]
we infer
\[
 f(\xi)=\Gamma(\alpha)\frac{1}{(\i\xi+\rho)^\alpha}.
\]
Since the left hand side of (\ref{eq:laplace}) equals $\frac{1}{\Gamma(\alpha)}f(\xi)$, the assertion follows.
\end{proof}

Next, we draw the connection from our fractional integral to the one
used in the literature.
\begin{thm}
\label{t:fracint} For all $\rho,\alpha>0$, $f\in L_{\rho}^{2}(\mathbb{R};H)$,
we have
\[
\partial_{0,\rho}^{-\alpha}f(t)=\int_{-\infty}^{t}\tfrac{1}{\Gamma(\alpha)}(t-s)^{\alpha-1}f(s)\,\mathrm{d}s.
\]
\end{thm}

\begin{proof}
For the proof we set $g \coloneqq \big(t \in \R \mapsto \frac{1}{\Gamma(\alpha)} t^{\alpha - 1} \chi_{\R_{>0}}(t)\big)$.
Then $g \in L_\rho^1(\R)$.
For $f \in L_\rho^2(\R; H)$ we have by Youngs convolution inequality
\begin{equation*}
        (e^{-\rho m} g) * (e^{-\rho m} f) = e^{-\rho m} (g*f) \in L^2(\R; H).
\end{equation*}
Using the convolution property of the Fourier transform we obtain
\begin{equation*}
        \sqrt{2\pi} \mathcal L_\rho g \cdot \mathcal L_\rho f = \mathcal L_\rho(g * f).
\end{equation*}
Using Lemma \ref{lem:FL-trafo-frac} we compute
\begin{align*}
        \partial_{0, \rho}^{-\alpha} f &= \mathcal L_\rho^* \left(\frac{1}{im + \rho}\right)^\alpha \mathcal L_\rho f
        \\
        &= \mathcal L_\rho^* \big(\sqrt{2\pi} \mathcal L_\rho g \cdot \mathcal L_\rho f\big)
        \\
        &= \mathcal L_\rho^* \mathcal L_\rho (g*f)
        \\
        &= \int_{-\infty}^{(\cdot)}\tfrac{1}{\Gamma(\alpha)}\big((\cdot)-s\big)^{\alpha-1}f(s)\,\mathrm{d}s.\qedhere
\end{align*}
\end{proof}

\begin{cor}
\label{c:poly} Let $\rho,\alpha>0$. Then for all $t\in\mathbb{R}$,
we have for $h\in H$
\[
(\partial_{0,\rho}^{-\alpha}\chi_{\mathbb{R}_{>0}}h)(t)=\begin{cases}
\tfrac{1}{\Gamma(\alpha+1)}t^{\alpha}h, & t>0,\\
0, & t\leq0.
\end{cases}
\]
\end{cor}

\begin{proof}
We use Theorem \ref{t:fracint} and obtain for $t\in\mathbb{R}$
\begin{align*}
(\partial_{0,\rho}^{-\alpha}\chi_{\mathbb{R}_{>0}}h)(t) & =\int_{-\infty}^{t}\tfrac{1}{\Gamma(\alpha)}(t-s)^{\alpha-1}\chi_{\mathbb{R}_{>0}}(s)\:h\:\mathrm{d}s\\
 & =\int_{\mathbb{R}}\chi_{\mathbb{R}_{>0}}(t-s)\chi_{\mathbb{R}_{>0}}(s)\tfrac{1}{\Gamma(\alpha)}(t-s)^{\alpha-1}\:\mathrm{d}s\;h.
\end{align*}
Thus, if $t>0$, we obtain
\begin{align*}
(\partial_{0,\rho}^{-\alpha}\chi_{\mathbb{R}_{>0}}h)(t) & =\int_{0}^{t}\tfrac{1}{\Gamma(\alpha)}(t-s)^{\alpha-1}\:\mathrm{d}s\:h\\
 & =\int_{0}^{t}\tfrac{1}{\Gamma(\alpha)}s^{\alpha-1}\:\mathrm{d}s\:h=\tfrac{1}{\Gamma(\alpha)}\tfrac{1}{\alpha}t^{\alpha}\:h.
\end{align*}
For $t\leq0$, we infer $\chi_{\mathbb{R}_{>0}}(t-s)\chi_{\mathbb{R}_{>0}}(s) = 0$ for $s \in \R$, i.e.\ $(\partial_{0,\rho}^{-\alpha}\chi_{\mathbb{R}_{>0}}h)(t) = 0$.
\end{proof}

\begin{rem}
It seems to be hard to determine analog formulas for the case $\rho<0$, although the operator $\partial_{0,\rho}^{-\alpha}$ for $\rho<0,\alpha>0$ is bounded.
The reason for this is that the corresponding multiplier $(\i m+\rho)^{-\alpha}$ is not defined in $0$ and has a jump there.
In particular, it cannot be extended to an analytic function on some right half plane of $\mathbb{C}$.
This, however, corresponds to the causality or anticausality of the operator $\partial_{0,\rho}^{-\alpha}$ by a Paley-Wiener result (\cite{PW} or \cite[19.2 Theorem]{Futh1})
and hence, we cannot expect to get a  convolution formula as in the case $\rho>0$.
\end{rem}

\section{A reformulation of classical Riemann\textendash Liouville and Caputo
differential equations\label{sec:RLC1}}

There are two main concepts of fractional differentiation (or integration).
In this section we shall start to identify both of these notions as being part of the \emph{same} solution theory, related to the spectral representation construction above.
We study Riemann\textendash Liouville and Caputo differential equations and their respective integral equations.
This section is only an introduction for the sections to come, where a comprehensive theory regarding well-posedness of fractional differential equations for a wide range of right-hand sides, is provided.
In fact it will turn out that Caputo differential equations can be readily rephrased with the notions developed in Section \ref{sec:rho_calc}.
We shall see that for Riemann\textendash Liouville differential equations some more theory has to be put in place.

To start off, we recall the Caputo differential equation. In \cite{Diethelm2003},
the author treated the following initial value problem of Caputo type
for $1\geq\alpha>0$:
\begin{align*}
D_{*}^{\alpha}y(t) & =f(t,y(t))\quad (t>0)\\
y(0) & =y_{0},
\end{align*}
where a solution $y$ is continuous at zero and $y_{0}\in\mathbb{C}^{n}$ is a given initial value; $f\colon\mathbb{R}_{>0}\times\mathbb{C}^{n}\to\mathbb{C}^{n}$
is continuous, satisfying
\begin{equation}\label{eq:Lipschitz}
\left|f(t,y_{1})-f(t,y_{2})\right|\leq c|y_{1}-y_{2}|
\end{equation}
for some $c\geq0$ and all $y_{1},y_{2}\in\mathbb{C}^{n},$ $t>0$.
For definiteness, we shall also assume that
\begin{equation}\label{eq:growth}
(t\mapsto f(t,0))\in L^2_{\rho_0}(\R_{>0};\C^n)
\end{equation}
for some $\rho_0\in \R$. In order to circumvent
discussions of how to \emph{interpret} the initial condition, we shall
rather put \cite[Equation (6)]{Diethelm2003} into the perspective
of the present exposition. In fact, this equation reads
\begin{equation}
y(t)=y_{0}+\frac{1}{\Gamma(\alpha)}\int_{0}^{t}(t-s)^{\alpha-1}f(s,y(s)) \d s\quad(t>0).\label{eq:IVP_Cap}
\end{equation}

First of all, we remark that in contrast to the setting in the previous
section, the differential equation just discussed `lives' on $\mathbb{R}_{>0}$,
only.
To this end we put
\begin{align*}
\tilde{f}\colon\mathbb{R}\times\mathbb{C}^{n}\to\mathbb{C}^{n},\quad & (t,y)\mapsto\chi_{\R_{> 0}}(t)f(t,y),
\end{align*}
with the apparent meaning that $\tilde{f}$ vanishes for negative
times $t$. We note that by \eqref{eq:Lipschitz} and \eqref{eq:growth} it follows that
\[
 L^2_\rho(\R) \ni y\mapsto (t\mapsto \tilde{f}(t,y(t)))\in L^2_\rho(\R)
\]
is a well-defined Lipschitz continuous mapping for all $\rho\geq \rho_0$.
Obviously, \eqref{eq:IVP_Cap} is equivalent to
\begin{equation}
y(t)=y_{0}\chi_{\R_{> 0}}(t)+\frac{1}{\Gamma(\alpha)}\int_{-\infty}^{t}(t-s)^{\alpha-1}\tilde{f}(s,y(s)) \d s\quad(t>0),\label{eq:IVP_CapR}
\end{equation}
which in turn can be (trivially) stated for \emph{all} $t\in\mathbb{R}$.
Next, we present the desired reformulation of equation \eqref{eq:IVP_CapR}.

\begin{thm}
\label{thm:Int_Cap}Let $\rho>\max\{0,\rho_0\}$. Assume that $y\in L_{\rho}^{2}(\mathbb{R})$.
Then the following statements are equivalent:

\begin{enumerate}

\item[(i)] $y(t)=y_{0}\chi_{\R_{> 0}}(t)+\frac{1}{\Gamma(\alpha)}\int_{-\infty}^{t}(t-s)^{\alpha-1}\tilde{f}(s,y(s)) \d s$ for almost every $t\in \R$,

\item[(ii)] $y=\partial_{0,\rho}^{-\alpha}\tilde{f}(\cdot,y(\cdot))+y_{0}\chi_{\R_{> 0}}$,

\item[(iii)] $\partial_{0,\rho}^{\alpha}(y-y_{0}\chi_{\R_{> 0}})=\tilde{f}(\cdot,y(\cdot))$.

\end{enumerate}
\end{thm}

\begin{proof}
The assertion follows trivially from Theorem \ref{t:fracint}.
\end{proof}

\begin{rem}
\leavevmode
\begin{enumerate}
\item
For a real valued-function $g:\R_{>0}\times\R^n\to\R^n$ we may consider the Caputo differential equation with $f:\R_{>0}\times\C^n\to\C^n,
(t,z)\mapsto g(t,\Re(z))$.
\item
In particular, we have shown in Theorem \ref{thm:Int_Cap} that the notions of so-called mild and strong solutions coincide.
\end{enumerate}
\end{rem}

Next we introduce Riemann\textendash Liouville differential equations.
Using the exposition in \cite{Samko1993}, we want to discuss
the Riemann\textendash Liouville fractional differential equation
given by
\begin{align*}
\frac{d^{\alpha}}{dx^{\alpha}}y(x) & =f(x,y(x)),\\
\left.\frac{d^{\alpha-1}}{dx^{\alpha-1}}y(x)\right|_{x=0+} & =y_{0},
\end{align*}
where as before $f$ satisfies \eqref{eq:Lipschitz} and \eqref{eq:growth} and $y_{0}\in\mathbb{R},$
and $\alpha\in(0,1]$. Again, not hinging on too much of an interpretation
of this equation, we shall rather reformulate the equivalent integral
equation related to this initial value problem. According to \cite[Chapter 42]{Samko1993}
this initial value problem can be formulated as
\[
y(t)=y_{0}\frac{t^{\alpha-1}}{\Gamma(\alpha)}+\frac{1}{\Gamma(\alpha)}\int_{0}^{t}(t-s)^{\alpha-1}f(s,y(s)) \d s\quad(t>0).
\]
We abbreviate $g_\beta(t)\coloneqq \frac{1}{\Gamma(\beta+1)}t^{\beta}\chi_{\R_{>0}}(t)$ for $t,\beta\in\R$.
For $\alpha>1/2$ we have $g_{\alpha-1}\in L_\rho^2(\R;H)$.
Let us assume that $\alpha>1/2$.
Invoking the cut-off function $\chi_{\R_{> 0}}$ and defining $\tilde{f}$ as
before, we may provide a reformulation of the Riemann\textendash Liouville equation on the space $L_\rho^2(\R;H)$ by
\[
y=g_{\alpha-1}y_0+\partial_{0,\rho}^{-\alpha}f(\cdot,y(\cdot)),\qquad y\in L_\rho^2(\R;H).
\]
By a formal calculation and when applying Corollary \ref{c:poly}, i.e.\ $\partial_{0,\rho}^{-\alpha}\chi_{\R_{>0}}y_0=g_\alpha y_0$, we would obtain
\[
g_{\alpha-1}y_0=\partial_{0,\rho}g_\alpha y_0=\partial_{0,\rho}\partial_{0,\rho}^{-\alpha}\chi_{\R_{>0}}y_0
=\partial_{0,\rho}^{-\alpha}\partial_{0,\rho}\chi_{\R_{>0}}y_0=\partial_{0,\rho}^{-\alpha}y_0\delta_0,
\]
where $\partial_{0,\rho}\chi_{\R_{>0}}y_0$ is, when understood distributionally, the delta function $y_0\delta_0$ and we could reformulate the Riemann\textendash Liouville equation by
\begin{equation}\label{eq:RL}
\partial_{0,\rho}^{\alpha}y=y_{0}\delta_{0}+\tilde{f}(\cdot,y(\cdot)).
\end{equation}
However, the calculation indicates that we have to extend the $L_\rho^2(\R;H)$ calculus to understand Riemann\textendash Liouville differential equations.
This will be done in the coming sections.

\section{Extra- and interpolation spaces}

We begin to define extra- and interpolation spaces associated with
the fractional derivative $\partial_{0,\rho}^{\alpha}$ for $\rho\ne0,$
$\alpha\in\mathbb{R}$. Since by definition
\[
\partial_{0,\rho}^{\alpha}=\mathcal{L}_{\rho}^{\ast}\left(\i m+\rho\right)^{\alpha}\mathcal{L}_{\rho},
\]
we will define the extra- and interpolation spaces in terms of the
multiplication operators $(\i m+\rho)^{\alpha}$ on $L^{2}(\mathbb{R};H).$

\begin{defn*}
Let $\rho\ne0$. For each $\alpha\in\mathbb{R}$ we define the space
\[
H^{\alpha}(\i m+\rho)\coloneqq\left\{ f\in L^1_{\mathrm{loc}}(\mathbb{R};H)\,;\,\intop_{\mathbb{R}}\|(\i t+\rho)^{\alpha}f(t)\|_{H}^{2}\text{ d}t<\infty\right\}
\]
and equip it with the natural inner product
\[
\langle f,g\rangle_{H^{\alpha}(\i m+\rho)}\coloneqq\intop_{\mathbb{R}}\langle(\i t+\rho)^{\alpha}f(t),(\i t+\rho)^{\alpha}g(t)\rangle_{H}\text{ d}t
\]
for each $f,g\in H^{\alpha}(\i m+\rho).$
\end{defn*}

We shall use $X\hookrightarrow Y$ to denote the mapping $X\ni x\mapsto x\in Y$,
if $X\subseteq Y$ (under a canonical identification, which will
always be obvious from the context).

\begin{lem}
\label{lem:Sobolev_chain_mult}For $\rho\ne0$ and $\alpha\in\mathbb{R}$
the space $H^{\alpha}(\i m+\rho)$ is a Hilbert space. Moreover, for
$\beta>\alpha$ we have
\[
j_{\beta\to\alpha}:H^{\beta}(\i m+\rho)\hookrightarrow H^{\alpha}(\i m+\rho)
\]
where the embedding is dense and continuous with $\|j_{\beta\to\alpha}\|\leq|\rho|^{\alpha-\beta}$.
\end{lem}

\begin{proof}
Note that $H^{\alpha}(\i m+\rho)=L^{2}(\mu;H)$, where $\mu$ is the
Lebesgue measure on $\mathbb{R}$ weighted with the function $t\mapsto|\ii t+\rho|^{2\alpha}$.
Thus, $H^{\alpha}(\i m+\rho)$ is a Hilbert space by the Fischer\textendash Riesz
theorem. Let now $\beta>\alpha$ and $f\in H^{\beta}(\i m+\rho)$.
Then
\[
\intop_{\mathbb{R}}\|(\i t+\rho)^{\alpha}f(t)\|_{H}^{2}\text{ d}t=\intop_{\mathbb{R}}(t^{2}+\rho^{2})^{\alpha-\beta}\|(\i t+\rho)^{\beta}f(t)\|_{H}^{2}\text{ d}t\leq\left(\rho^{2}\right)^{\alpha-\beta}\|f\|_{H^{\beta}(\i m+\rho)}^{2},
\]
which proves the continuity of the embedding $j_{\beta\to\alpha}$
and the asserted norm estimate. The density follows, since $C_{c}^{\infty}(\mathbb{R};H)$
lies dense in $H^{\gamma}(\i m+\rho)$ for each $\gamma\in\mathbb{R}.$
\end{proof}

\begin{defn*}
Let $\rho\ne0$ and $\alpha\in\mathbb{R}.$ We consider the space
\[
W_{\rho}^{\alpha}(\mathbb{R};H)\coloneqq\left\{ u\in L^2_{\rho}(\mathbb{R};H)\,;\,\mathcal{L}_{\rho}u\in H^{\alpha}(\i m+\rho)\right\}
\]
equipped with the inner product
\[
\langle u,v\rangle_{\rho,\alpha}\coloneqq\langle\mathcal{L}_{\rho}u,\mathcal{L}_{\rho}v\rangle_{H^{\alpha}(\i m+\rho)}
\]
and set $H_{\rho}^{\alpha}(\mathbb{R};H)$ as its completion with
respect to the norm induced by $\langle\cdot,\cdot\rangle_{\rho,\alpha}.$
\end{defn*}

\begin{lem}
\label{lem:Sobolev_chain_fractional}Let $\rho\ne0$.
\begin{enumerate}
\item For $\alpha\geq0$ we have that $H_{\rho}^{\alpha}(\mathbb{R};H)=W_{\rho}^{\alpha}(\mathbb{R};H)=\dom(\partial_{0,\rho}^{\alpha}).$
\item The operator
\[
\mathcal{L}_{\rho}:W_{\rho}^{\alpha}(\mathbb{R};H)\subseteq H_{\rho}^{\alpha}(\mathbb{R};H)\to H^{\alpha}(\i m+\rho)
\]
has a unique unitary extension, which will again be denoted by $\mathcal{L}_{\rho}.$
\item For $\alpha,\beta\in\mathbb{R}$ with $\beta>\alpha$ we have that
\[
\iota_{\beta\to\alpha}:H_{\rho}^{\beta}(\mathbb{R};H)\hookrightarrow H_{\rho}^{\alpha}(\mathbb{R};H)
\]
is continuous and dense with $\|\iota_{\beta\to\alpha}\|\leq|\rho|^{\alpha-\beta}.$
\item For each $\beta>0$ and $\alpha\in\mathbb{R}$ the operator
\[
\partial_{0,\rho}^{\beta}:H_{\rho}^{\beta+|\alpha|}(\mathbb{R};H)\subseteq H_{\rho}^{\alpha}(\mathbb{R};H)\to H_{\rho}^{\alpha-\beta}(\mathbb{R};H)
\]
has a unique unitary extension, which will again be denoted by $\partial_{0,\rho}^{\beta}.$
\end{enumerate}
\end{lem}

\begin{proof}
$\,$
\begin{enumerate}
\item Let $\alpha \geq 0$. For $u \in H^\alpha(\i m+\rho)$, i.e. $u \in L_\textrm{loc}^1(\R; H)$ and $(\i m+\rho)^\alpha u \in L^2(\R; H)$, we infer that $u \in L^2(\R; H)$.
It follows that $u \in \dom((\i m+\rho)^\alpha)$.
Hence $H^\alpha(\i m+\rho) = \dom((\i m+\rho)^\alpha)$.
Moreover,
\begin{align*}
u\in W_{\rho}^{\alpha}(\mathbb{R};H) & \Leftrightarrow u\in L^2_{\rho}(\mathbb{R};H)\wedge\mathcal{L}_{\rho}u\in H^{\alpha}(\i m+\rho)\\
 & \Leftrightarrow u\in L^2_{\rho}(\mathbb{R};H)\wedge\mathcal{L}_{\rho}u\in\dom\left((\i m+\rho)^{\alpha}\right)\\
 & \Leftrightarrow u\in\dom(\partial_{0,\rho}^{\alpha}),
\end{align*}
by Example \ref{exa:deffrac}. Moreover, since
\[
\mathcal{L}_{\rho}:W_{\rho}^{\alpha}(\mathbb{R};H)\to H^{\alpha}(\i m+\rho)
\]
is unitary, we infer that $W_{\rho}^{\alpha}(\mathbb{R};H)$ is complete
with respect to $\|\cdot\|_{\rho,\alpha}=\|\mathcal{L}_{\rho}\cdot\|_{H^{\alpha}(\i m+\rho)},$
and thus $H_{\rho}^{\alpha}(\mathbb{R};H)=W_{\rho}^{\alpha}(\mathbb{R};H).$
\item Obviously,
\[
\mathcal{L}_{\rho}:W_{\rho}^{\alpha}(\mathbb{R};H)\subseteq H_{\rho}^{\alpha}(\mathbb{R};H)\to H^{\alpha}(\i m+\rho)
\]
is isometric by the definition of the norm on $H_{\rho}^{\alpha}(\mathbb{R};H).$
Moreover, its range is dense, since $\mathcal{L}_{\rho}^{\ast}\varphi\in W_{\rho}^{\alpha}(\mathbb{R};H)$
for each $\varphi\in C_{c}^{\infty}(\mathbb{R};H)$ and thus, $C_{c}^{\infty}(\mathbb{R};H)\subseteq\mathcal{L}_{\rho}\left[W_{\rho}^{\alpha}(\mathbb{R};H)\right].$
Hence, the continuous extension of $\mathcal{L}_{\rho}$ to $H_{\varrho}^{\alpha}(\mathbb{R};H)$
is onto and, thus, unitary.
\item Since $\iota_{\beta\to\alpha}=\mathcal{L}_{\rho}^{\ast}j_{\beta\to\alpha}\mathcal{L}_{\rho},$
the assertion follows from Lemma \ref{lem:Sobolev_chain_mult}.
\item Since,
\begin{align*}
(\i m+\rho)^{\beta}:H^{\alpha}(\i m+\rho) & \to H^{\alpha-\beta}(\i m+\rho)\\
f & \mapsto\left(t\mapsto(\i t+\rho)^{\beta}f(t)\right)
\end{align*}
is obviously unitary, we infer that for $u\in H_{\rho}^{\beta+|\alpha|}(\mathbb{R};H)$
\begin{align*}
\|\partial_{0,\rho}^{\beta}u\|_{\rho,\alpha-\beta} & =\|\mathcal{L}_{\rho}\partial_{0,\rho}^{\beta}u\|_{H^{\alpha-\beta}(\i m+\rho)}\\
 & =\|(\i m+\rho)^{\beta}\mathcal{L}_{\rho}u\|_{H^{\alpha-\beta}(\i m+\rho)}\\
 & =\|\mathcal{L}_{\rho}u\|_{H^{\alpha}(\i m+\rho)}\\
 & =\|u\|_{\rho,\alpha},
\end{align*}
which shows that $\partial_{0,\rho}^{\beta}$ is an isometry. Moreover,
for $\varphi\in C_{c}^{\infty}(\mathbb{R};H)$, we have that $(\i m+\rho)^{\gamma}\varphi\in C_{c}^{\infty}(\mathbb{R};H)$
for all $\gamma\in\mathbb{R}$ and thus, in particular $\mathcal{L}_{\rho}^{\ast}(\i m+\rho)^{-\beta}\varphi\in\bigcap_{\gamma\in\mathbb{R}}H_{\rho}^{\gamma}(\mathbb{R};H)\subseteq H_{\rho}^{\beta+|\alpha|}(\mathbb{R};H).$
Next,
\[
\partial_{0,\rho}^{\beta}\mathcal{L}_{\rho}^{\ast}(\i m+\rho)^{-\beta}\varphi=\mathcal{L}_{\rho}^{\ast}(\i m+\rho)^{\beta}\mathcal{L}_{\rho}\mathcal{L}_{\rho}^{\ast}(\i m+\rho)^{-\beta}\varphi=\mathcal{L}_{\rho}^{\ast}\varphi
\]
and thus, $\mathcal{L}_{\rho}^{\ast}[C_{c}^{\infty}(\mathbb{R};H)]\subseteq\partial_{0,\rho}^{\beta}[H_{\rho}^{\beta+|\alpha|}(\mathbb{R};H)]$.
Since $C_{c}^{\infty}(\mathbb{R};H)$ is dense in $H^{\alpha-\beta}(\i m+\rho)$,
we infer that $\mathcal{L}_{\rho}^{\ast}[C_{c}^{\infty}(\mathbb{R};H)]$
is dense in $H_{\rho}^{\alpha-\beta}(\mathbb{R};H)$ and thus, $\partial_{0,\rho}^{\beta}$
has dense range. This completes the proof.\hfill{}$\qedhere$
\end{enumerate}
\end{proof}
We conclude this section by providing an alternative perspective to
elements lying in $H_{\rho}^{\alpha}(\mathbb{R};H)$ for some $\alpha\in\mathbb{R}$
(with a particular focus on $\alpha<0$). In particular, we aim for
a definition of a support for those elements which coincides with
the usual support of $L^{2}$ functions in the case $\alpha\geq0$.

\begin{lem}
\label{lem:sigma-1_unitary}
Let $\rho\ne0$ and $\alpha\in\mathbb{R}$.
Then
\begin{align*}
\sigma_{-1}:W_{\rho}^{\alpha}(\mathbb{R};H)\subseteq H_{\rho}^{\alpha}(\mathbb{R};H) & \to H_{-\rho}^{\alpha}(\mathbb{R};H)\\
f & \mapsto\left(t\mapsto f(-t)\right)
\end{align*}
extends to a unitary operator.
Moreover, for $f \in H_\rho^\alpha(\R; H)$ we have
\[
\mathcal L_{-\rho} \sigma_{-1} f=\sigma_{-1} \mathcal L_\rho f\qquad\mathrm{and}\qquad \mathcal L_{-\rho}^* \sigma_{-1}f=\sigma_{-1} \mathcal L_\rho^* f.
\]
\end{lem}

\begin{proof}
For $f\in W_{\rho}^{\alpha}(\mathbb{R};H)$ we have that
\[
\mathcal{L}_{-\rho}\sigma_{-1}f=\sigma_{-1}\mathcal{L}_{\rho}f
\]
and hence,
\begin{align*}
\intop_{\mathbb{R}}\left\|(\i t-\rho)^{\alpha}\left(\mathcal{L}_{-\rho}\sigma_{-1}f\right)(t)\right\|_H^{2}\text{ d}t & =\intop_{\mathbb{R}}\left(t^{2}+\rho^{2}\right)^{\alpha}\left\|\left(\mathcal{L}_{\rho}f\right)(-t)\right\|_H^{2}\text{ d}t\\
 & =\intop_{\mathbb{R}}\left(t^{2}+\rho^{2}\right)^{\alpha}\left\|\left(\mathcal{L}_{\rho}f\right)(t)\right\|_H^{2}\text{ d}t=\|f\|_{H_{\rho}^{\alpha}(\mathbb{R};H)}^{2},
\end{align*}
which proves the isometry of $\sigma_{-1}.$ Moreover, $\sigma_{-1}$
has dense range, since $\sigma_{-1}[W_{\rho}^{\alpha}(\mathbb{R};H)]=W_{-\rho}^{\alpha}(\mathbb{R};H).$
Hence, $\sigma_{-1}$ extends to a unitary operator.
The equality $\mathcal L_{-\rho} \sigma_{-1} f=\sigma_{-1} \mathcal L_\rho f$ holds for $f\in H_\rho^\alpha(\R;H)$,
since $W_\rho^\alpha(\R;H)$ is dense in its completion $H_\rho^\alpha(\R;H)$.
\end{proof}

\begin{prop}
\label{prop:distributions}Let $\rho\ne0,$ $\alpha\in\mathbb{R}$
and $f\in H_{\rho}^{\alpha}(\mathbb{R};H).$ Then
\[
\langle f,\cdot\rangle:C_{c}^{\infty}(\mathbb{R};H)\to\mathbb{C}
\]
given by
\[
\langle f,\varphi\rangle\coloneqq\intop_{\mathbb{R}}\langle\mathcal{L}_{\rho}f(t),\mathcal{L}_{-\rho}\varphi(t)\rangle_{H}\,\mathrm{d}t
\]
defines a distribution. Moreover, for $f\in H_{\rho}^{\alpha}(\mathbb{R};H)$ and $\varphi\in C_c^\infty(\R;H)$
we have
\[
\langle f,\varphi\rangle=\langle f,\partial_{0,\rho}^{-\alpha}\e^{2\rho m}\sigma_{-1}\partial_{0,\rho}^{-\alpha}\sigma_{-1}\varphi\rangle_{\rho,\alpha}.
\]
In particular, for $\alpha=0$
\[
\langle f,\varphi\rangle=\intop_{\mathbb{R}}\langle f(t),\varphi(t)\rangle_{H}\,\mathrm{d}t.
\]
Note that the operator $\partial_{0,\rho}^{-\alpha}\e^{2\rho m}\sigma_{-1}\partial_{0,\rho}^{-\alpha}\sigma_{-1}$ maps $H_{-\rho}^{-\alpha}(\R; H)$ to $H_{\rho}^{\alpha}(\R; H)$ unitarily.
\end{prop}

\begin{proof}
Let $f\in H_{\rho}^{\alpha}(\mathbb{R};H).$ We first prove that the
expression $\langle f,\cdot\rangle$ is indeed a distribution. Due
to Lemma \ref{lem:Sobolev_chain_fractional}(c) it suffices to prove
this for $f\in H_{\rho}^{-k}(\mathbb{R};H)$ for some $k\in\mathbb{N}.$
Indeed, if $f\in H_{\varrho}^{-k}(\mathbb{R};H),$ then we know that
\[
\left(t\mapsto(\i t+\rho)^{-k}\left(\mathcal{L}_{\rho}f\right)(t)\right)\in L^{2}(\mathbb{R};H)
\]
and hence, for $\varphi\in C_{c}^{\infty}(\mathbb{R};H)$ we obtain
using H\"older's inequality and the fact that $\mathcal L_{-\rho}\varphi^{(k)} = (\i m+\rho)^k\mathcal L_{-\rho} \varphi$
\begin{align*}
|\langle f,\varphi\rangle| & \leq\intop_{\mathbb{R}}|\langle(\i t+\rho)^{-k}(\mathcal{L}_{\rho}f)(t),(-\i t+\rho)^{k}\left(\mathcal{L}_{-\rho}\varphi\right)(t)\rangle_{H}|\text{ d}t\\
 & \leq\left\|\mathcal{L}_{\rho}f\right\|_{H^{-k}(\i m+\rho)}\|\mathcal{L}_{-\rho}\left(\varphi^{(k)}\right)\|_{L^{2}(\mathbb{R};H)}\\
 & \leq\left\|\mathcal{L}_{\rho}f\right\|_{H^{-k}(\i m+\rho)}\left(\:\intop_{\spt\varphi}\e^{2\rho t}\text{ d}t\right)^{\frac{1}{2}}\|\varphi^{(k)}\|_{\infty},
\end{align*}
which proves that $\langle f,\cdot\rangle$ is indeed a distribution.
Next, we prove the asserted formula. For this, we note the following
elementary equality
\[
\sigma_{-1}\mathcal{L}_{\rho}\varphi=\mathcal{L}_{\rho}e^{2\rho m}\sigma_{-1}\varphi
\]
for $\varphi\in L_{\rho}^{2}(\mathbb{R};H)$. Let $f\in H_{\rho}^{\alpha}(\mathbb{R};H)$
and compute
\begin{align*}
\langle f,\varphi\rangle & =\langle\mathcal{L}_{\rho}f,\mathcal{L}_{-\rho}\varphi\rangle_{L^{2}(\mathbb{R};H)}\\
 & =\langle\mathcal{L}_{\rho}f,\sigma_{-1}\mathcal{L}_{\rho}\sigma_{-1}\varphi\rangle_{L^{2}(\mathbb{R};H)}\\
 & =\langle(\ii m+\rho)^{\alpha}\mathcal{L}_{\rho}f,(-\ii m+\rho)^{-\alpha}\sigma_{-1}\mathcal{L}_{\rho}\sigma_{-1}\varphi\rangle_{L^{2}(\mathbb{R};H)}\\
 & =\langle(\ii m+\rho)^{\alpha}\mathcal{L}_{\rho}f,\sigma_{-1}(\ii m+\rho)^{-\alpha}\mathcal{L}_{\rho}\sigma_{-1}\varphi\rangle_{L^{2}(\mathbb{R};H)}\\
 & =\langle(\ii m+\rho)^{\alpha}\mathcal{L}_{\rho}f,\sigma_{-1}\mathcal{L}_{\rho}\partial_{0,\rho}^{-\alpha}\sigma_{-1}\varphi\rangle_{L^{2}(\mathbb{R};H)}\\
 & =\langle(\ii m+\rho)^{\alpha}\mathcal{L}_{\rho}f,\mathcal{L}_{\rho}\e^{2\rho m}\sigma_{-1}\partial_{0,\rho}^{-\alpha}\sigma_{-1}\varphi\rangle_{L^{2}(\mathbb{R};H)}\\
 & =\langle(\ii m+\rho)^{\alpha}\mathcal{L}_{\rho}f,(\ii m+\rho)^{\alpha}\mathcal{L}_{\rho}\partial_{0,\rho}^{-\alpha}\e^{2\rho m}\sigma_{-1}\partial_{0,\rho}^{-\alpha}\sigma_{-1}\varphi\rangle_{L^{2}(\mathbb{R};H)}\\
 & =\langle f,\partial_{0,\rho}^{-\alpha}\e^{2\rho m}\sigma_{-1}\partial_{0,\rho}^{-\alpha}\sigma_{-1}\varphi\rangle_{\rho,\alpha}
\end{align*}
for each $\varphi\in C_{c}^{\infty}(\mathbb{R};H).$ In particular,
in the case $\alpha=0$ we obtain
\[
\langle f,\varphi\rangle=\langle f,\e^{2\rho m}\varphi\rangle_{\rho,0}=\intop_{\mathbb{R}}\langle f(t),\varphi(t)\rangle_{H}\text{ d}t.\text{\tag*{\qedhere}}
\]
\end{proof}

\begin{rem}
The latter proposition shows that $\bigcup_{\rho\ne0,\alpha\in\mathbb{R}}H_{\rho}^{\alpha}(\mathbb{R};H)\subseteq\mathcal{D}(\mathbb{R};H)'.$
In particular, the support of an element in $H_{\rho}^{\alpha}(\mathbb{R};H)$
is then well-defined by
\[
\bigcap\{\R\setminus U; U\subseteq\R\;\mathrm{open},\; \forall\varphi\in C_c^\infty(U;H):\langle f,\varphi\rangle=0\},
\]
and the second part of the latter proposition
shows, that it coincides with the usual $L^{2}$-support if $\alpha\geq0.$
Moreover, we can now compare elements in $H_{\rho}^{\alpha}(\mathbb{R};H)$
and $H_{\mu}^{\beta}(\mathbb{R};H)$ by saying that those elements
are equal if they are equal as distributions. We shall further elaborate
on this matter in Proposition \ref{prop:char_distributions}. In particular,
we shall show that $f\mapsto\langle f,\cdot\rangle$ is injective.
We shall also mention that the notation $\langle f,\varphi\rangle$ is
justified, as it does \emph{not} depend on $\rho$ nor $\alpha$.
\end{rem}

\begin{example}
\label{exa:disder}Let $f\in L_{\rho}^{2}(\mathbb{R};H)$. Then, by
definition, $\partial_{0,\rho}f\in H_{\rho}^{-1}(\mathbb{R};H)$.
We shall compute the action of $\partial_{0,\rho}f$ as a distribution.
For this let $\varphi\in C_{c}^{\infty}(\mathbb{R};H)$ and we compute
with the formula outlined in Proposition \ref{prop:distributions}
for $\alpha=-1$:
\begin{align*}
\langle\partial_{0,\rho}f,\varphi\rangle & =\langle\partial_{0,\rho}f,\partial_{0,\rho}\e^{2\rho m}\sigma_{-1}\partial_{0,\rho}\sigma_{-1}\varphi\rangle_{\rho,-1}\\
 & =\langle(\ii m+\rho)^{-1}\mathcal{L}_{\rho}\partial_{0,\rho}f,(\ii m+\rho)^{-1}\mathcal{L}_{\rho}\partial_{0,\rho}\e^{2\rho m}\sigma_{-1}\partial_{0,\rho}\sigma_{-1}\varphi\rangle_{L^2(\R;H)}\\
 & =\langle\mathcal{L}_{\rho}f,\mathcal{L}_{\rho}\e^{2\rho m}\sigma_{-1}\partial_{0,\rho}\sigma_{-1}\varphi\rangle_{L^{2}(\mathbb{R};H)}\\
 & =-\langle\mathcal{L}_{\rho}f,\mathcal{L}_{\rho}\e^{2\rho m}\partial_{0,\rho}\varphi\rangle_{L^{2}(\mathbb{R};H)}\\
 & =-\langle f,\e^{2\rho m}\varphi'\rangle_{L_{\rho}^{2}(\mathbb{R};H)}\\
 & =-\int_{\mathbb{R}}\langle f(t),\varphi'(t)\rangle_{H}\ \text{d}t.
\end{align*}
Thus, $\partial_{0,\rho}f$ coincides with the distibutional derivative
of $L_{\rho}^{2}(\mathbb{R};H)$ functions.
\end{example}

\begin{lem}\label{lem:der_spt}
Let $\alpha \in \R$.
Let $H_\rho^\infty(\R;H)\coloneqq\bigcap_{k\in\N}H_\rho^k(\R;H)$ for $\rho \neq 0$.
\begin{enumerate}
        \item\label{lem:der_spt_a} Let $\varphi\in C_c^\infty(\R;H)$.
        For all $\rho>0$ we have $\partial_\rho^{\alpha}\varphi\in C^\infty(\R;H)\cap H_\rho^\infty(\R;H)$ and $\inf\spt\partial_\rho^\alpha\varphi \geq\inf\spt\varphi$.
        For $\rho,\mu>0$, $\alpha\in\R$ we have $\partial_{0,\rho}^\alpha\varphi = \partial_{0,\mu}^\alpha\varphi$.
        \item\label{lem:der_spt_b} Let $\alpha\in\R$ and $\mu,\rho\neq 0$.
        Let $\psi\in C^\infty(\R;H)\cap H_\rho^\infty(\R;H)\cap H_\mu^\infty(\R;H)$.
        Then there is $(\varphi_n)_{n\in\N}\in C_c^\infty(\R;H)^\N$ s.t. $\varphi_n\to\psi$ for $n\to\infty$ in $H_\rho^\alpha(\R;H)$ and $H_\mu^\alpha(\R;H)$
        and $\spt(\varphi_n)\subseteq \spt(\psi)$ for $n\in\N$.
\end{enumerate}
\end{lem}
\begin{proof}
\ref{lem:der_spt_a}:
Let $\alpha\in\R$.
Let $\mu,\rho>0$.
For $\alpha>0$ it holds that $\partial_\rho^\alpha = \partial_\rho^{\alpha-\lceil\alpha\rceil}\partial_\rho^{\lceil\alpha\rceil}$ and $\partial_\rho^{\lceil\alpha\rceil}\varphi=\varphi^{(\lceil\alpha\rceil)}=\partial_\mu^{\lceil\alpha\rceil}\varphi\in C_c^\infty(\R;H)$ and $\alpha-\lceil\alpha\rceil < 0$.
Thus we may assume that $\alpha<0$.
By Theorem \ref{t:fracint} we have $\partial_{0,\rho}^{\alpha}\varphi = \partial_{0,\mu}^{\alpha}\varphi$ and $\inf\spt\partial_{0,\rho}^{\alpha}\varphi > -\infty$.
From $\varphi\in H_\rho^\infty(\R;H)$ we deduce $\partial_{0,\rho}^{\alpha}\varphi\in H_\rho^\infty(\R;H)$.
\\
\ref{lem:der_spt_b}:
Let $k\in\mathbb{N}$ with $k\geq\alpha$.
We choose a sequence $\left(\chi_{n}\right)_{n\in\mathbb{N}}$ in $C_{c}^{\infty}(\mathbb{R})$ such that $\spt\chi_{n}\subseteq[-n-1,n+1]$, $\chi_{n}=1$ on $[-n,n]$ and
\begin{equation*}
\sup\left\{ \|\chi_{n}^{(j)}\|_{\infty}\,;\,j\in\{0,\ldots,k\},n\in\mathbb{N}\right\} <\infty.
\end{equation*}
Set $\varphi_n\coloneqq\chi_n\psi \in C_c^\infty(\R;H)$.
Then $\spt(\varphi_n)\subseteq\spt(\psi)$ for $n\in\N$.
Since $H_\nu^k(\R;H)$ is dense and continuously embedded into $H_\nu^\alpha(\R;H)$ ($\nu\neq 0$),
it suffices to show that $\varphi_n\to\psi$ ($n\to\infty$) in $H_\rho^k(\R;H)$ and $H_\mu^k(\R;H)$.
Indeed, by the product rule, the choice of $\chi_n$ and dominated convergence we obtain
\begin{align*}
        \varphi_n^{(k)} = \sum_{j=0}^k \chi_n^{(k-j)}\psi^{(j)} = \chi_n\psi^{(k)} + \sum_{j=0}^{k-1} \chi_n^{(k-j)}\psi^{(j)} \to \psi^{(k)}
\end{align*}
for $n\to\infty$ in $L_\rho^2(\R;H)$ and in $L_\mu^2(\R;H)$.
\end{proof}

\begin{lem}
\label{lem:testfunctions_dense}Let $\rho\ne0$ and $\alpha\in\mathbb{R}.$
Then $C_{c}^{\infty}(\mathbb{R};H)$ is dense in $H_{\rho}^{\alpha}(\mathbb{R};H).$
\end{lem}

\begin{proof}
We first note that it suffices to prove the assertion for $\rho>0,$
since the operator $\sigma_{-1}$ from Lemma \ref{lem:sigma-1_unitary}
leaves $C_{c}^{\infty}(\mathbb{R};H)$ invariant.
It is well known that $C_c^\infty(\R;H)$ is dense in $L_\rho^2(\R;H)$.
We have $\partial_{0,\rho}^\alpha f\in L_\rho^2(\R;H)$.
Let $(\psi_n)_{n\in\N}\in C_c^\infty(\R;H)^\N$ with $\psi_n\to \partial_{0,\rho}^\alpha f$ ($n\to\infty$) in $L_\rho^2(\R;H)$.
By Lemma \ref{lem:der_spt}\ref{lem:der_spt_a} we have $\partial_{0,\rho}^{-\alpha}\psi_n\in C^\infty(\R;H)\cap H_\rho^\infty(\R;H)$
and by Lemma \ref{lem:der_spt}\ref{lem:der_spt_b} we find $(\varphi_n)_{n\in\N}\in C_c^\infty(\R;H)^\N$ with $\norm{\partial_{0,\rho}^{-\alpha}\psi_n-\varphi_n}_{\rho, \alpha}\to 0$ ($n\to\infty$).
Then
\[
\norm{f-\varphi_n}_{\rho,\alpha}\leq \norm{f-\partial_{0,\rho}^{-\alpha}\psi_n}_{\rho, \alpha}
+\norm{\partial_{0,\rho}^{-\alpha}\psi_n-\varphi_n}_{\rho,\alpha}\to 0\qquad (n\to\infty).\qedhere
\]
\end{proof}

With this result at hand, we can characterize those distributions,
which belong to $H_{\rho}^{\alpha}(\mathbb{R};H)$ for some $\alpha\in\mathbb{R},\rho\ne0$,
in the following way.
\begin{prop}
\label{prop:char_distributions}Let $\psi\in\mathcal{D}(\mathbb{R};H)'$
and $\alpha\in\mathbb{R},\rho\ne0.$ Then, there exists $f\in H_{\rho}^{\alpha}(\mathbb{R};H)$
such that
\[
\psi(\varphi)=\langle f,\varphi\rangle\quad(\varphi\in C_{c}^{\infty}(\mathbb{R};H))
\]
in the sense of Propositions \ref{prop:distributions} if and only
if there is $C\geq0$ such that
\[
|\psi(\varphi)|\leq C\|\varphi\|_{-\rho,-\alpha}
\]
for each $\varphi\in C_{c}^{\infty}(\mathbb{R};H).$
\end{prop}

\begin{proof}
Assume first that there is $f\in H_{\rho}^{\alpha}(\mathbb{R};H)$
representing $\psi.$ Then we estimate
\begin{align*}
|\psi(\varphi)| & =|\langle f,\varphi\rangle|\\
 & =\left|\intop_{\mathbb{R}}\langle\mathcal{L}_{\rho}f(t),\mathcal{L}_{-\rho}\varphi(t)\rangle_{H}\text{ d}t\right|\\
 & =\left|\intop_{\mathbb{R}}\langle(\i t+\rho)^{\alpha}\mathcal{L}_{\rho}f(t),(-\i t+\rho)^{-\alpha}\mathcal{L}_{-\rho}\varphi(t)\rangle_{H}\text{ d}t\right|\\
 & \leq\|\mathcal{L}_{\rho}f\|_{H^{\alpha}(\i m+\rho)}\|\mathcal{L}_{-\rho}\varphi\|_{H^{-\alpha}(\i m-\rho)}\\
 & =\|f\|_{\rho,\alpha}\|\varphi\|_{-\rho,-\alpha}
\end{align*}
for each $\varphi\in C_{c}^{\infty}(\mathbb{R};H).$
Let $C \geq 0$ such that $\psi$ satisfies for $\varphi \in C_c^\infty(\R; H)$
\begin{equation*}
        \abs{\psi(\varphi)} \leq C \norm{\varphi}_{-\rho, -\alpha}.
\end{equation*}
The operator
$A \coloneqq \partial_{0, \rho}^{-\alpha} \e^{2\rho m} \sigma_{-1} \partial_{0, \rho}^{-\alpha} \sigma_{-1}: H_{-\rho}^{-\alpha}(\R; H) \to H_\rho^\alpha(\R; H)$
(cf.\ Proposition \ref{prop:distributions}) is unitary.
Thus for $\varphi \in C_c^\infty(\R; H)$
\begin{equation*}
        \abs{\psi(A^{-1}\varphi)} \leq C \norm{A^{-1} \varphi}_{-\rho, -\alpha} = C \norm{\varphi}_{\rho, \alpha}.
\end{equation*}
Moreover, $C_c^\infty(\R; H) \subseteq H_\rho^\alpha(\R; H)$ is dense.
Thus $\psi(A^{-1} \cdot)$ can be extended continuously to $H_\rho^\alpha(\R; H)$.
By the Riesz representation theorem, there is a $f \in H_\rho^\alpha(\R; H)$ such that for $\varphi \in H_\rho^\alpha(\R; H)$
\begin{equation*}
        \psi(A^{-1}\varphi) = \left<f, \varphi\right>_{\rho, \alpha}.
\end{equation*}
By Theorem \ref{prop:distributions} we have for $\varphi \in C_c^\infty(\R; H)$
\begin{equation*}
        \psi(\varphi) = \psi(A^{-1} A \varphi) = \left<f, A\varphi\right>_{\rho, \alpha} = \left<f, \varphi\right>.\qedhere
\end{equation*}
\end{proof}

In the next proposition, we shall also obtain the announced uniqueness
statement, that is, the injectivity of the mapping $f\mapsto\langle f,\cdot\rangle$.
\begin{prop}
\label{prop:independence}Let $\alpha\in\mathbb{R}$ and $\mu,\rho>0.$
Moreover, let $f\in H_{\rho}^{\alpha}(\mathbb{R};H)$ and $g\in H_{\mu}^{\alpha}(\mathbb{R};H).$
Then the following statements are equivalent:

\begin{enumerate}

\item[(i)] $f=g$ in the sense of distributions, i.e., for each $\varphi\in C_{c}^{\infty}(\mathbb{R};H)$
we have that
\[
\intop_{\mathbb{R}}\langle\mathcal{L}_{\rho}f(t),\mathcal{L}_{-\rho}\varphi(t)\rangle_{H}\,\mathrm{d}t=\intop_{\mathbb{R}}\langle\mathcal{L}_{\mu}g(t),\mathcal{L}_{-\mu}\varphi(t)\rangle_{H}\,\mathrm{d}t.
\]

\item[(ii)] $\partial_{0,\rho}^{\alpha}f=\partial_{0,\mu}^{\alpha}g$
as functions in $L^1_{\mathrm{loc}}(\mathbb{R};H)$.

\item[(iii)] There is a sequence $(\varphi_{n})_{n\in\mathbb{N}}$
in $C_{c}^{\infty}(\mathbb{R};H)$ with $\varphi_{n}\to f$ in $H_{\rho}^{\alpha}(\mathbb{R};H)$
and $\varphi_{n}\to g$ in $H_{\mu}^{\alpha}(\mathbb{R};H)$ as $n\to\infty$.

\end{enumerate}
\end{prop}

\begin{proof}
(i)$\Rightarrow$(ii):
Let $\psi\in C_c^\infty(\R;H)$ and $\tilde\psi \coloneqq \sigma_{-1}\partial_\rho^\alpha\sigma_{-1}\psi = \sigma_{-1}\partial_\mu^\alpha\sigma_{-1}\psi$.
Then by Lemma \ref{lem:der_spt}\ref{lem:der_spt_a} $\tilde\psi\in C^\infty(\R;H)\cap H_{-\rho}^\infty(\R;H)\cap H_{-\mu}^\infty(\R;H)$.
By Lemma \ref{lem:der_spt}\ref{lem:der_spt_b} there's $(\varphi_n)_{n\in\N}\in C_c^\infty(\R;H)^\N$ with $\varphi_n\to\tilde\psi$ ($n\to\infty$) in $H_{-\rho}^{-\alpha}(\R;H)$
and in $H_{-\mu}^{-\alpha}(\R;H)$.
Thus
\begin{align*}
        \int_\R\langle\partial_{0,\rho}^\alpha f(t),\psi(t)\rangle_H \d t &= \langle\partial_{0,\rho}^{\alpha}f,\e^{2\rho m}\psi\rangle_{\rho,0}
        \\
        &= \langle f,\partial_{0,\rho}^{-\alpha}\e^{2\rho m}\sigma_{-1}\partial_{0,\rho}^{-\alpha}\sigma_{-1}(\sigma_{-1}\partial_{0,\rho}^{\alpha}\sigma_{-1}\psi)\rangle_{\rho,\alpha}
        \\
        &= \lim_{n\to\infty} \langle f,\partial_{0,\rho}^{-\alpha}\e^{2\rho m}\sigma_{-1}\partial_{0,\rho}^{-\alpha}\sigma_{-1}\varphi_n\rangle_{\rho,\alpha}
        \\
        &= \lim_{n\to\infty}\langle f,\varphi_n\rangle
        \\
        &= \lim_{n\to\infty} \langle g,\varphi_n\rangle
        \\
        &= \lim_{n\to\infty} \langle f,\partial_{0,\mu}^{-\alpha}\e^{2\mu m}\sigma_{-1}\partial_{0,\mu}^{-\alpha}\sigma_{-1}\varphi_n\rangle_{\mu,\alpha}
        \\
        &= \int_\R\langle\partial_{0,\mu}^\alpha f(t),\psi(t)\rangle_H \d t.
\end{align*}
(ii) $\Rightarrow$ (iii): Define $\tilde{f}_{n}\coloneqq\chi_{[-n,n]}\cdot\partial_{0,\rho}^{\alpha}f=\chi_{[-n,n]}\cdot\partial_{0,\mu}^{\alpha}g$
for $n\in\mathbb{N}$. Without loss of generality let $\rho<\mu.$
Take a function $\psi_{n}\in C_{c}^{\infty}(\mathbb{R};H)$ with $\spt\psi_{n}\subseteq[-n,n]$
such that
\[
\|\tilde{f}_{n}-\psi_{n}\|_{\rho,0}\leq\frac{1}{n}\e^{(\rho-\mu)n}.
\]
Then, we estimate
\begin{align*}
   \|\tilde{f}_{n}-\psi_{n}\|_{\mu,0}^{2}
   &=
   \intop_{-n}^{n}\|\tilde{f}_{n}(t)-\psi_{n}(t)\|_{H}^{2}\e^{-2\mu t}\text{ d}t
   =
   \intop_{-n}^{n}\|\tilde{f}_{n}(t)-\psi_{n}(t)\|_{H}^{2}\e^{-2\rho t}\e^{2\left(\rho-\mu\right)t}\text{ d}t
\\
   &\leq
   \|\tilde{f}_{n}-\psi_{n}\|_{\rho,0}^{2}\e^{2(\mu-\rho)n}\leq\frac{1}{n^2}.
\end{align*}
Hence $\psi_n\to\partial_{0,\rho}^\alpha f=\partial_{0,\mu}^\alpha g$ in $L_\rho^2(\R;H)$ and in $L_\mu^2(\R;H)$ by the triangle inequality and dominated convergence.
We set $\tilde\varphi_n \coloneqq \partial_{0,\rho}^{-\alpha}\psi_n = \partial_{0,\mu}^{-\alpha}\psi_n \in C^\infty(\R;H)\cap H_\rho^\infty(\R;H)\cap H_\mu^\infty(\R;H)$.
Then $\tilde\varphi_n \to f$ and $\tilde\varphi_n \to g$ in $H_\rho^\alpha(\R;H)$ and in $H_\mu^\alpha(\R;H)$ respectively.
We use Lemma \ref{lem:der_spt}\ref{lem:der_spt_b} and choose a sequence $(\varphi_n)_{n\in\N}\in C_c^\infty(\R;H)^\N$ with $\norm{\tilde\varphi_n-\varphi_n}_{\rho,\alpha}\to 0$.
Then
\begin{align*}
        &\norm{f-\varphi_n}_{\rho,\alpha}\leq\norm{f-\tilde\varphi_n}_{\rho,\alpha} + \norm{\tilde\varphi_n - \varphi_n}_{\rho,\alpha}\to 0 &&(n\to\infty),
        \\
        &\norm{g-\varphi_n}_{\mu,\alpha}\leq\norm{g-\tilde\varphi_n}_{\mu,\alpha} + \norm{\tilde\varphi_n - \varphi_n}_{\mu,\alpha}\to 0 &&(n\to\infty).
\end{align*}
(iii) $\Rightarrow$ (i): Let $(\varphi_{n})_{n\in\mathbb{N}}$ be
a sequence in $C_{c}^{\infty}(\mathbb{R};H)$ such that $\varphi_{n}\to f$
and $\varphi_{n}\to g$ in $H_{\rho}^{\alpha}(\mathbb{R};H)$ and
$H_{\mu}^{\alpha}(\mathbb{R};H)$, respectively. Let $\varphi\in C_{c}^{\infty}(\mathbb{R};H).$
Then we have according to Proposition \ref{prop:distributions}
\begin{align*}
\intop_{\mathbb{R}}\langle\mathcal{L}_{\rho}f(t),\mathcal{L}_{-\rho}\varphi(t)\rangle_{H}\,\mathrm{d}t & =\langle f,\varphi\rangle\\
 & =\langle f,\partial_{0,\rho}^{-\alpha}\e^{2\rho m}\sigma_{-1}\partial_{0,\rho}^{-\alpha}\sigma_{-1}\varphi\rangle_{\rho,\alpha}\\
 & =\lim_{n\to\infty}\langle\varphi_{n},\partial_{0,\rho}^{-\alpha}\e^{2\rho m}\sigma_{-1}\partial_{0,\rho}^{-\alpha}\sigma_{-1}\varphi\rangle_{\rho,\alpha}\\
 & =\lim_{n\to\infty}\langle\varphi_{n},\varphi\rangle\\
 & =\lim_{n\to\infty}\langle\varphi_{n},\partial_{0,\mu}^{-\alpha}\e^{2\mu m}\sigma_{-1}\partial_{0,\mu}^{-\alpha}\sigma_{-1}\varphi\rangle_{\mu,\alpha}\\
 & =\langle g,\varphi\rangle,\\
 &= \intop_\R \left<\mathcal L_\mu f(t), \mathcal L_{-\mu} \varphi(t)\right> \d t.
\end{align*}
which completes the proof.
\end{proof}

\section{A unified solution theory \textendash{} well-posedness and causality
of fractional differential equations}

We are now able to study abstract fractional differential equations
of the form
\[
\partial_{0,\rho}^{\alpha}u=F(u).
\]
In order to obtain well-posedness of the latter problem, we need to
restrict the class of admissible right-hand sides $F$ in the latter
equation.

\begin{defn*}
Let $\rho_{0}>0$ and $\beta,\gamma\in\mathbb{R}.$ We call a function
$F:\dom(F)\subseteq\bigcap_{\rho\geq\rho_{0}}H_{\rho}^{\beta}(\mathbb{R};H)\to\bigcap_{\rho\geq\rho_{0}}H_{\rho}^{\gamma}(\mathbb{R};H)$
\emph{eventually $(\beta,\gamma)$-Lipschitz continuous, }if $\dom(F)\supseteq C_{c}^{\infty}(\mathbb{R};H)$
and there exists $\nu\geq\rho_{0}$ such that for each $\rho\ge\nu$
the function $F$ has a Lipschitz continuous extension
\[
F_{\rho}:H_{\rho}^{\beta}(\mathbb{R};H)\to H_{\rho}^{\gamma}(\mathbb{R};H)
\]
satisfying $\sup_{\rho\geq\nu}|F_{\rho}|_{\mathrm{Lip}}<\infty$.
Moreover, we call $F$ \emph{eventually $(\beta,\gamma)$-contracting,
}if $F$ is eventually $(\beta,\gamma)$-Lipschitz continuous and
$\limsup_{\rho\to\infty}|F_{\rho}|_{\mathrm{Lip}}<1.$ Here, we denote
by $|\cdot|_{\mathrm{Lip}}$ the smallest Lipschitz constant of a
Lipschitz continuous function:
\[
\abs{F_\rho}_\mathrm{lip} \coloneqq\sup_{f, g \in H_\rho^\beta(\R; H), \; f \neq g}
\frac{\norm{F_\rho(f) - F_\rho(g)}_{\rho, \gamma}}{\norm{f-g}_{\rho, \beta}}.
\]
\end{defn*}

Note that by Lemma \ref{lem:testfunctions_dense}, any eventually
Lipschitz continuous function is densely defined. Thus, the Lipschitz continuous
extension \emph{$F_{\rho}$} is \emph{unique}.

\begin{rem}
\label{rem:eveLip}(a) If $f\in H_{\rho}^{\beta}(\mathbb{R};H)$ and
$g\in H_{\mu}^{\beta}(\mathbb{R};H)$ generate the same distribution,
we have that
\[
F_{\rho}(f)=F_{\mu}(g).
\]

Indeed, by Proposition \ref{prop:independence} there exists a sequence
$(\varphi_{n})_{n\in\mathbb{N}}$ in $C_{c}^{\infty}(\mathbb{R};H)$
with $\varphi_{n}\to f$ and $\varphi_{n}\to g$ in $H_{\rho}^{\beta}(\mathbb{R};H)$
and $H_{\mu}^{\beta}(\mathbb{R};H)$, respectively. We infer that
\[
F_{\rho}(f) =\lim_{n\to\infty}F(\varphi_{n})\qquad\textrm{and}\qquad
F_{\mu}(g) =\lim_{n\to\infty}F(\varphi_{n})
\]
with convergence in $H_{\rho}^{\gamma}(\mathbb{R};H)$ and $H_{\mu}^{\gamma}(\mathbb{R};H)$ respectively.
Consequently
\[
\partial_{0,\rho}^{\gamma}F_{\rho}(f)\leftarrow\partial_{0,\rho}^{\gamma}F(\varphi_{n})=\partial_{0,\mu}^{\gamma}F(\varphi_{n})\to\partial_{0,\mu}^{\gamma}F_{\mu}(g)
\]
with convergence in $L_\rho^2(\R; H)$ and hence almost everywhere for a suitable subsequence of $(\varphi_n)_{n \in \N}$.
The assertion follows from Proposition \ref{prop:independence}.

(b) We shall need the following elementary observation later on.
Let $F$ be evenutally $(\beta,\gamma)$-Lipschitz continuous, $\alpha\in\mathbb{R}$.
Let $\rho\geq\rho_{0}$.
Then
\[
\tilde{F}:C_{c}^{\infty}(\mathbb{R};H)\ni\varphi\mapsto F_{\rho}(\partial_{0,\rho}^{\alpha}\varphi)
\]
is eventually $(\beta+\alpha,\gamma)$-Lipschitz continuous.
Indeed, the assertion follows from part (a)  and
\[
\norm{\tilde F (f)-\tilde F (g)}_{\mu, \gamma}
\leq \abs{F_{\mu}}_\mathrm{Lip} \norm{\partial_{0, \mu}^\alpha f - \partial_{0, \mu}^\alpha g}_{\mu, \beta}
= \abs{F_{\mu}}_\mathrm{Lip} \norm{f - g}_{\mu, \alpha + \beta},
\]
for $\mu \geq \nu$, $f, g \in C_c^\infty(\R; H)$.
\end{rem}

\begin{thm}
\label{thm:sol_theory_contraction}Let $\alpha>0,\beta\in\mathbb{R},\rho_{0}>0$
and $F:\dom(F)\subseteq\bigcap_{\rho\geq\rho_{0}}H_{\rho}^{\beta}(\mathbb{R};H)\to\bigcap_{\rho\geq\rho_{0}}H_{\rho}^{\beta-\alpha}(\mathbb{R};H)$
be eventually $(\beta,\beta-\alpha)$-contracting. Then there exists
$\nu\geq\rho_{0}$ such that for each $\rho\geq\nu$ there is a unique
$u_{\rho}\in H_{\rho}^{\beta}(\mathbb{R};H)$ satisfying
\begin{equation}
\partial_{0,\rho}^{\alpha}u_{\rho}=F_{\rho}(u_{\rho}).\label{eq:eq_frac}
\end{equation}
\end{thm}

\begin{proof}
This is a simple consequence of the contraction mapping theorem. Indeed,
choosing $\nu\geq\rho_{0}$ large enough, such that $|F_{\rho}|_{\mathrm{Lip}}<1$
for each $\rho \geq \nu$, we obtain that
\[
\partial_{0,\rho}^{-\alpha}F_{\rho}:H_{\rho}^{\beta}(\mathbb{R};H)\to H_{\rho}^{\beta}(\mathbb{R};H)
\]
is a strict contraction, since $\partial_{0,\rho}^{-\alpha}:H_{\rho}^{\beta-\alpha}(\mathbb{R};H)\to H_{\rho}^{\beta}(\mathbb{R};H)$
is unitary by Lemma \ref{lem:Sobolev_chain_fractional}. Hence, the
mapping $\partial_{0,\rho}^{-\alpha}F_{\rho}$ admits a unique fixed
point $u_{\rho}\in H_{\rho}^{\beta}(\mathbb{R};H),$ which is equivalent
to $u_{\rho}$ being a solution of (\ref{eq:eq_frac}).
\end{proof}

\begin{cor}
\label{cor:sol_theory_Lipschitz}Let $\alpha>0,\beta\in\mathbb{R},\rho_{0}>0$
and $F:\dom(F)\subseteq\bigcap_{\rho\geq\rho_{0}}H_{\rho}^{\beta}(\mathbb{R};H)\to\bigcap_{\rho\geq\rho_{0}}H_{\rho}^{\beta-\gamma}(\mathbb{R};H)$
for some $\gamma\in[0,\alpha[$ be eventually $(\beta,\beta-\gamma)$-Lipschitz continuous.
Then there exists $\nu\geq\rho_{0}$ such that for each $\rho\geq\nu$
there is a unique $u_{\rho}\in H_{\rho}^{\beta}(\mathbb{R};H)$ satisfying
\[
\partial_{0,\rho}^{\alpha}u_{\rho}=F_{\rho}(u_{\rho}).
\]
\end{cor}

\begin{proof}
It suffices to prove that $\iota_{\beta-\gamma\to\beta-\alpha}\circ F$
is eventually $(\beta,\beta-\alpha)$-contracting by Theorem \ref{thm:sol_theory_contraction}.
Let $\nu \geq \rho$, s.t. for $\rho \geq \nu$, $F_\rho$ exists.
Then for $\rho \geq \nu$
\[
\left|\iota_{\beta-\gamma\to\beta-\alpha}\circ F_{\rho}\right|_{\mathrm{Lip}}\leq\|\iota_{\beta-\gamma\to\beta-\alpha}\||F_{\rho}|_{\mathrm{Lip}}\leq\rho^{\gamma-\alpha}|F_{\rho}|_{\mathrm{Lip}}
\]
by Lemma \ref{lem:Sobolev_chain_fractional}. Since $|F_{\rho}|_{\mathrm{Lip}}$
is bounded in $\rho$ on $[\nu, \infty[$ by assumption, we infer
\[
\limsup_{\rho\to\infty}\left|\iota_{\beta-\gamma\to\beta-\alpha}\circ F_{\rho}\right|_{\mathrm{Lip}}=0<1.\tag*{\qedhere}
\]
\end{proof}

Next, we want to show that the solution $u_{\rho}$ of (\ref{eq:eq_frac})
is actually independent of the particular choice of $\rho$. For doing
so, we need the concept of causality, which will be addressed in the
next propositions.

\begin{lem}\label{lem:fracsupport}
Let $\rho > 0$, $\alpha \in \R$ and $a \in \R$.
Let $f \in H_\rho^\alpha(\R; H)$ with $\spt f \subseteq \R_{\geq a}$.
Then there is a sequence $(\varphi_n)_{n \in \N} \in C_c^\infty(\R; H)^\N$ with $\spt \varphi_n \subseteq \R_{\geq a}$ for $n \in \N$ and $\varphi_n \to f$ in $H_\rho^\alpha(\R; H)$ as $n \to \infty$.
\end{lem}
\begin{proof}
Let $(\tilde\psi_n)_{n \in \N} \in C_c^\infty(\R; H)^\N$ be such that $\tilde\psi_n \to \partial_{0,\rho}^\alpha f$ in $H_\rho^0(\R; H)$ as $n \to \infty$.
We may assume that $\spt\tilde\psi_n \subseteq \R_{> a}$.
We set $\psi_n\coloneqq\partial_{0,\rho}^{-\alpha}\tilde\psi_n$ for $n\in\N$.
Then $\psi_n \to f$ as $n\to\infty$ in $H_\rho^\alpha(\R; H)$ and $\inf\spt\psi_n>a$ by Lemma \ref{lem:der_spt}\ref{lem:der_spt_a}.
We use Lemma \ref{lem:der_spt}\ref{lem:der_spt_b} and pick a sequence $(\varphi_n)_{n\in\N}\in C_c^\infty(\R;H)^\N$ with $\spt(\varphi_n)\subseteq\spt(\psi_n)$ for $n\in\N$ and $\varphi_n-\psi_n\to 0$ in $H_\rho^\alpha(\R;H)$ when $n\to\infty$.
Then
\begin{equation*}
        \norm{\varphi_n - f}_{\rho,\alpha} \leq \norm{\varphi_n - \psi_n}_{\rho,\alpha} + \norm{\psi_n - f}_{\rho,\alpha}\to 0\qquad (n\to\infty).\qedhere
\end{equation*}
\end{proof}

\begin{prop}
\label{prop:causality_fractional}Let $f\in H_{\rho}^{\alpha}(\mathbb{R};H)$
for some $\alpha\in\mathbb{R},\rho>0.$ Assume that $\spt f\subseteq\mathbb{R}_{\geq a}$
for some $a\in\mathbb{R}.$ Then
\[
\spt\partial_{0,\rho}^{\beta}f\subseteq\mathbb{R}_{\geq a}
\]
for all $\beta\in\mathbb{R}.$
\end{prop}

\begin{proof}
Let $\varphi \in C_c^\infty(\R; H)$ with $\spt \varphi \subseteq \R_{<a}$.
By Lemma \ref{lem:fracsupport} we pick a sequence $(\varphi_n)_{n \in \N} \in C_c^\infty(\R; H)^\N$,
s.t.\ $\spt \varphi_n \subseteq \R_{\geq a}$ ($n\in\N$) and $\varphi_n \to f$ in $H_\rho^\alpha(\R; H)$.
Then $\spt \partial_{0, \rho}^\beta \varphi_n \subseteq \R_{\geq a}$ by Lemma \ref{lem:der_spt}\ref{lem:der_spt_a}.
By Proposition \ref{prop:distributions} we have.
\begin{equation*}
        \left<\partial_{0, \rho}^\beta \varphi_n, \varphi\right> = \int_\R \left<\partial_{0, \rho}^\beta \varphi_n, \varphi\right>_H(t) \d t = 0.
\end{equation*}
Since $\partial_{0, \rho}^\beta$ is unitary, we have $\partial_{0, \rho}^\beta \varphi_n \to \partial_{0, \rho}^\beta f$ in $H_\rho^{\alpha - \beta}(\R; H)$.
We compute
\begin{align*}
        \left<\partial_{0, \rho}^\beta f, \varphi\right> & = \langle\partial_{0,\rho}^{\beta}f,\partial_{0,\rho}^{-(\alpha-\beta)}\e^{2\rho m}\sigma_{-1}\partial_{0,\rho}^{-(\alpha-\beta)}\sigma_{-1}\varphi\rangle_{\rho,\alpha-\beta}\\
        & = \lim_{n\to\infty} \langle\partial_{0,\rho}^{\beta}\varphi_n,\partial_{0,\rho}^{-(\alpha-\beta)}\e^{2\rho m}\sigma_{-1}\partial_{0,\rho}^{-(\alpha-\beta)}\sigma_{-1}\varphi\rangle_{\rho,\alpha-\beta}\\
        & = \lim_{n\to\infty} \left<\partial_{0, \rho}^\beta \varphi_n, \varphi\right>\\
        & = 0. \qedhere
\end{align*}
\end{proof}

The proof of the following theorem outlining causality of $\partial_{0,\rho}^{-\alpha}F_{\rho}$,
is in spirit similar to the approach in \cite[Theorem 4.5]{kalauch2014hilbert}.
However, one has to adopt the distributional setting and the (different)
definition of eventually Lipschitz continuity here accordingly.

\begin{thm}
\label{thm:causality}Let the assumptions of Theorem \ref{thm:sol_theory_contraction}
be satisfied. Then, for each $\rho\geq\nu,$ where $\nu$ is chosen
according to Theorem \ref{thm:sol_theory_contraction}, the mapping
\[
\partial_{0,\rho}^{-\alpha}F_{\rho}:H_{\rho}^{\beta}(\mathbb{R};H)\to H_{\rho}^{\beta}(\mathbb{R};H)
\]
is \emph{causal}, that is, for each $u,v\in H_{\rho}^{\beta}(\mathbb{R};H)$
satisfying $\spt(u-v)\subseteq\mathbb{R}_{\geq a}$ for some $a\in\mathbb{R},$
it holds that $\spt\left(\partial_{0,\rho}^{-\alpha}F_{\rho}(u)-\partial_{0,\rho}^{-\alpha}F_{\rho}(v)\right)\subseteq\mathbb{R}_{\geq a}.$
Here, the support is meant in the sense of distributions.
\end{thm}

\begin{proof}
First of all, we shall show the result for $u,v\in C_{c}^{\infty}(\mathbb{R};H)$.
So, let $u,v\in C_{c}^{\infty}(\mathbb{R};H)$ with $\spt(u-v)\subseteq\mathbb{R}_{\geq a}.$
Take $\varphi\in C_{c}^{\infty}(\mathbb{R};H)$ with $\spt\varphi\subseteq\mathbb{R}_{<a}.$
Let $\mu \geq \rho$.
Then $F_\rho(u) = F_\mu(u)$ and
\begin{align*}
        \left<\partial_{0,\rho}^{-\alpha}(F_\rho(u)-F_\rho(v)), \varphi\right> &= \left<\partial_{0,\mu}^{-\alpha}(F_\mu(u)-F_\mu(v)), \varphi\right>
        \\
        &= \left<\partial_{0,\mu}^{-\alpha}(F_\mu(u)-F_\mu(v)), \partial_{0, \mu}^{-\beta}\e^{2\mu m}\sigma_{-1}\partial_{0,\mu}^{-\beta}\sigma_{-1} \varphi\right>_{\mu,\beta}
        \\
        &= \left<F_\mu(u)-F_\mu(v), \partial_{0, \mu}^{-(\beta-\alpha)}\e^{2\mu m}\sigma_{-1}\partial_{0,\mu}^{-\beta}\sigma_{-1} \varphi\right>_{\mu,\beta-\alpha}
        \\
        &\leq \norm{F_\mu(u)-F_\mu(v)}_{\mu, \beta-\alpha} \norm{\partial_{0, \mu}^{-(\beta-\alpha)}\e^{2\mu m}\sigma_{-1}\partial_{0,\mu}^{-\beta}\sigma_{-1} \varphi}_{\mu, \beta-\alpha}
        \\
        &\leq \abs{F_\mu}_\mathrm{lip} \norm{u-v}_{\mu, \beta} \norm{\partial_{0,\mu}^{-\beta}\sigma_{-1} \varphi}_{\mu, 0}
\end{align*}
where we have used that $\partial_{0, \mu}^{-(\beta-\alpha)}\e^{2\mu m}\sigma_{-1}: H_\mu^0(\R; H) \to H_\mu^{\beta-\alpha}(\R; H)$ is unitary and $\varphi \in H_{-\mu}^{-\beta}(\R; H)$.
According to Proposition \ref{prop:causality_fractional} we have
that $\spt\partial_{0,\mu}^{-\beta}\sigma_{-1}\varphi\subseteq\mathbb{R}_{>-a}$
and hence, we compute
\[
\|\partial_{0,\mu}^{-\beta}\sigma_{-1}\varphi\|_{\mu,0}^{2}=\intop_{-a}^{\infty}\left\|\left(\partial_{0,\mu}^{-\beta}\sigma_{-1}\varphi\right)(t)\right\|_{H}^{2}\e^{-2\mu t}\text{ d}t=\intop_{0}^{\infty}\left\|\left(\partial_{0,\mu}^{-\beta}\sigma_{-1}\varphi\right)(t-a)\right\|_{H}^{2}\e^{-2\mu t}\text{ d}t\,\e^{2\mu a}.
\]
On the other hand
\begin{align*}
\|u-v\|_{\mu,\beta}^{2} & =\|\partial_{0,\mu}^{\beta}(u-v)\|_{\mu,0}^{2}\\
 & =\intop_{a}^{\infty}\|\partial_{0,\mu}^{\beta}(u-v)(t)\|_{H}^{2}\e^{-2\mu t}\text{ d}t\\
 & =\intop_{0}^{\infty}\left\|\partial_{0,\mu}^{\beta}(u-v)(t+a)\right\|_{H}^{2}\e^{-2\mu t}\text{ d}t\e^{-2\mu a}
\end{align*}
and consequently,
\begin{align*}
 & |F_{\mu}|_{\mathrm{Lip}}\|u-v\|_{\mu,\beta}\|\partial_{0,\mu}^{-\beta}\sigma_{-1}\varphi\|_{-\mu,0}\\
 & =|F_{\mu}|_{\mathrm{Lip}}\left(\intop_{0}^{\infty}\left\|\partial_{0,\mu}^{\beta}(u-v)(t+a)\right\|_{H}^{2}\e^{-2\mu t}\text{ d}t\right)^{\frac{1}{2}}\intop_{0}^{\infty}\left\|\left(\partial_{0,\mu}^{-\beta}\sigma_{-1}\varphi\right)(t-a)\right\|_{H}^{2}\e^{-2\mu t}\text{ d}t\to0\quad(\mu\to\infty),
\end{align*}
by dominated convergence. Summarizing, we have shown that $\spt(\partial_{0,\rho}^{-\alpha} F_\rho(u)-\partial_{0,\rho}^{-\alpha}F_\rho(v))\subseteq \R_{\geq a}$ for $u,v\in C_c^\infty(\R;H)$ satisfying $\spt(u-v)\subseteq \R_{\geq a}$.

Before we conclude the proof, we show that if $(w_{n})_{n\in\mathbb{N}}$
is a convergent sequence in $H_{\rho}^{\beta}(\mathbb{R};H)$ with
$\spt w_{n}\subseteq\mathbb{R}_{\geq a}$ for each $n\in\mathbb{N}$,
then its limit $w$ also satisfies $\spt w\subseteq\mathbb{R}_{\geq a}.$
For doing so, let $\varphi\in C_{c}^{\infty}(\mathbb{R};H)$ with
$\spt\varphi\subseteq\mathbb{R}_{<a}.$
Then
\[
\left<w, \varphi\right> = \left<w, \partial_{0,\rho}^{-\beta}\e^{2\rho m} \sigma_{-1}\partial_{0,\rho}^{-\beta}\sigma_{-1}\varphi\right>_{\rho, \beta}
= \lim_{n\to\infty} \left<w_n, \partial_{0,\rho}^{-\beta}\e^{2\rho m} \sigma_{-1}\partial_{0,\rho}^{-\beta}\sigma_{-1}\varphi\right>_{\rho, \beta}
= \lim_{n\to\infty} \left<w_n, \varphi\right> = 0.
\]
Finally, let $u,v\in H_{\rho}^{\beta}(\mathbb{R};H)$ with $\spt (u-v)\subseteq\mathbb{R}_{\geq a}$
According to Lemma \ref{lem:fracsupport} there is a sequence $(\varphi_n)_{n \in \N} \in C_c^\infty(\R; H)^\N$ with $\spt \varphi_n \subseteq \R_{\geq a}$
and $\varphi_n \to u-v$ in $H_\rho^\alpha(\R; H)$ as $n \to \infty$.
Let $(v_n)_{n\in\N}\in C_c^\infty(\R; H)$ with $v_n \to v$ in $H_\rho^\alpha(\R; H)$ as $n \to \infty$.
We set $u_n \coloneqq \varphi_n + v_n$.
Then $u_n \to u$ in $H_\rho^\alpha(\R; H)$ and $\spt(u_n-v_n)\subseteq\R_{\geq a}$.
By the already proved result for $C_{c}^{\infty}(\mathbb{R};H)$,
we infer that $\spt\left(\partial_{0,\rho}^{-\alpha}F_{\rho}(u_n)-\partial_{0,\rho}^{-\alpha}F_{\rho}(v_n)\right)\subseteq\mathbb{R}_{\geq a}$
for all $n\in\mathbb{N}$. Thus, letting $n\to\infty$, we obtain
$\spt\left(\partial_{0,\rho}^{-\alpha}F_{\rho}(u)-\partial_{0,\rho}^{-\alpha}F_{\rho}(v)\right)\subseteq\mathbb{R}_{\geq a}$, which shows the claim.
\end{proof}

Finally, we prove that our solution is independent of the particular
choice of the parameter $\rho>\nu$ in Theorem \ref{thm:sol_theory_contraction}.
The precise statement is as follows.

\begin{prop}
Let the assumptions of Theorem \ref{thm:sol_theory_contraction} be
satisfied and $\nu$ be chosen according to Theorem \ref{thm:sol_theory_contraction}.
Let $\tilde{\mu},\mu>\nu$ and $u_{\tilde{\mu}}\in H_{\tilde{\mu}}^{\beta}(\mathbb{R};H),\,u_{\mu}\in H_{\mu}^{\beta}(\mathbb{R};H)$
satisfying
\[
\partial_{0,\tilde{\mu}}^{\alpha}u_{\tilde{\mu}}=F_{\tilde{\mu}}(u_{\tilde{\mu}})\text{ and }\partial_{0,\mu}^{\alpha}u_{\mu}=F_{\mu}(u_{\mu}).
\]
Then $u_{\tilde{\mu}}=u_{\mu}$ as distributions in the sense of Proposition \ref{prop:distributions}.
\end{prop}

\begin{proof}
We note that it suffices to show $v_{\mu}\coloneqq\partial_{0,\mu}^{\beta}u_{\mu}=\partial_{0,\tilde{\mu}}^{\beta}u_{\tilde{\mu}}\eqqcolon v_{\tilde{\mu}}$
as $L^1_{\mathrm{loc}}(\mathbb{R};H)$ functions by Proposition \ref{prop:independence}.
We consider the function
\[
\tilde{F}:\dom(\tilde{F})\subseteq\bigcap_{\rho\geq\rho_{0}}H_{\rho}^{0}(\mathbb{R};H)\to\bigcap_{\rho\geq\rho_{0}}H_{\rho}^{\beta-\alpha}(\mathbb{R};H)
\]
given by
\begin{equation}
\tilde{F}(v)\coloneqq F(\partial_{0,\rho}^{-\beta}v)\quad(v\in \dom(\tilde{F}))\label{eq:F_tilde}
\end{equation}
with maximal domain $\dom(\tilde{F})=\{w\in\bigcap_{\rho\geq\rho_{0}}H_{\rho}^0(\mathbb{R};H)\,;\,\forall\rho\geq\rho_{0}:\partial_{0,\rho}^{-\beta}w\in \dom(F)\}.$
Note that the expression on the right hand side of (\ref{eq:F_tilde})
does not depend on the particular choice of $\rho\geq\rho_{0}$ by
Proposition \ref{prop:independence}.
Clearly, $\tilde{F}$ is eventually $(0,\beta-\alpha)$-contracting
(see also Remark \ref{rem:eveLip}(b)) and
\[
\tilde{F}_{\rho}=F_{\rho}(\partial_{0,\rho}^{-\beta}(\cdot))\quad(\rho\geq\rho_{0}).
\]
In particular,
\[
\partial_{0,\mu}^{\alpha-\beta}v_{\mu}=\partial_{0,\mu}^{\alpha}u_{\mu}=F_{\mu}(u_{\mu})=\tilde{F}_{\mu}(v_{\mu})
\]
and analogously
\[
\partial_{0,\tilde{\mu}}^{\alpha-\beta}v_{\tilde{\mu}}=\tilde{F}_{\tilde{\mu}}(v_{\tilde{\mu}}).
\]

Let now $a\in\mathbb{R}$ and assume without loss of generality that
$\mu<\tilde{\mu}.$
We note that $\spt(v_{\tilde \mu} - \chi_{\R_\leq a} v_{\tilde \mu}) \subseteq \chi_{\R_{\geq a}}$.
We obtain, applying Theorem \ref{thm:causality},
that
\[
\chi_{\mathbb{R}_{\leq a}}v_{\tilde{\mu}}=\text{\ensuremath{\chi}}_{\mathbb{R}_{\leq a}}\partial_{0,\tilde{\mu}}^{\beta-\alpha}\tilde{F}_{\tilde{\mu}}(v_{\tilde{\mu}})=\text{\ensuremath{\chi}}_{\mathbb{R}_{\leq a}}\partial_{0,\tilde{\mu}}^{\beta-\alpha}\tilde{F}_{\tilde{\mu}}(\chi_{\mathbb{R}_{\leq a}}v_{\tilde{\mu}}).
\]
Now, since $\chi_{\mathbb{R}_{\leq a}}v_{\tilde{\mu}}\in L_{\mu}^{2}(\mathbb{R};H)\cap L_{\tilde{\mu}}^{2}(\mathbb{R};H)$,
we infer that
\[
\chi_{\mathbb{R}_{\leq a}}v_{\tilde{\mu}}=\text{\ensuremath{\chi}}_{\mathbb{R}_{\leq a}}\partial_{0,\tilde{\mu}}^{\beta-\alpha}\tilde{F}_{\tilde{\mu}}(\chi_{\mathbb{R}_{\leq a}}v_{\tilde{\mu}})=\text{\ensuremath{\chi}}_{\mathbb{R}_{\leq a}}\partial_{0,\mu}^{\beta-\alpha}\tilde{F}_{\mu}(\chi_{\mathbb{R}_{\leq a}}v_{\tilde{\mu}}),
\]
i.e. $\chi_{\mathbb{R}_{\leq a}}v_{\tilde{\mu}}$ is a fixed point
of $\text{\ensuremath{\chi}}_{\mathbb{R}_{\leq a}}\partial_{0,\mu}^{\beta-\alpha}\tilde{F}_{\mu}$.
However, since $\chi_{\mathbb{R}_{\leq a}}v_{\mu}$ is also a fixed
point of this mapping, which is strictly contractive, we derive
\[
\chi_{\mathbb{R}_{\leq a}}v_{\tilde{\mu}}=\chi_{\mathbb{R}_{\leq a}}v_{\mu}
\]
and since $a\in\mathbb{R}$ was arbitrary, the assertion follows.
\end{proof}

\section{Riemann\textendash Liouville and Caputo differential equations revisited\label{sec:RLC2}}

In this section, we shall consider the differential equations introduced
in Section \ref{sec:RLC1} and prove their well-posedness and causality.
First of all, we gather some results ensuring the Lipschitz continuity
property needed to apply either of the well-posedness theorems presented
in the previous section.
As in Section \ref{sec:RLC1} we fix $\alpha\in(0,1]$.

\begin{prop}
\label{prop:Lips}Let $\rho_0>0$, $n\in\mathbb{N}$, $y_{0}\in\mathbb{R}^{n}$,
$f\colon\mathbb{R}_{>0}\times\mathbb{C}^{n}\to\mathbb{C}^{n}$ continuous.
Assume there exists $c\geq0$ such that for all $y_{1}, y_{2}\in\mathbb{R}^{n}, t>0$
we have
\begin{align*}
|f(t,y_{1})-f(t,y_{2})| & \leq c|y_{1}-y_{2}|.
\end{align*}
Moreover, we assume that
\[
 (t\mapsto f(t,0))\in L^2_{\rho_0}(\R_{>0};\C^n).
\]

Define $\tilde{f}\colon\mathbb{R}\times\mathbb{C}^{n}\to\mathbb{C}^{n}$
by
\[
\tilde{f}(t,y)\coloneqq \begin{cases}
                         f(t,y) & \mbox{ if } t>0,\\
                         0 &\mbox{ else}.
                        \end{cases}
\]
Then the mapping $F\colon C_{c}^{\infty}(\mathbb{R};\mathbb{C}^{n})\to C(\mathbb{R};\mathbb{C}^{n})$
given by
\[
F(\varphi)(t)\coloneqq\tilde{f}(t,\varphi(t)+y_{0})\quad(\varphi\in C_c^\infty(\R;\C^n) ,t\in\mathbb{R})
\]
is eventually $(0,0)$-Lipschitz continuous.
\end{prop}

\begin{proof}
Let $\rho\geq \rho_0$. In order to prove that $F$ attains values in $L^2_\rho(\R;\C^n)$, we shall show $F(0)\in L_{\rho}^{2}(\mathbb{R};\mathbb{C}^{n})$ first. For this we compute
 \begin{align*}
\int_{\mathbb{R}}|F(0)(t)|^{2}\e^{-2\rho t}\text{ d}t & =\int_{\R_{>0}}|f(t,y_{0})|^{2}\e^{-2\rho t}\text{ d}t\\
 & \leq 2\left(\int_{\mathbb{R}_{>0}}|f(t,y_{0})-f(t,0)|^{2}\e^{-2\rho t}\text{ d}t+\int_{\mathbb{R}_{>0}}|f(t,0)|^{2}\e^{-2\rho t}\text{ d}t\right)\\
 &\leq 2 \left(c^2|y_0|^{2} \frac{1}{2\rho} + |f(\cdot,0)|_{L^2_{\rho_0}(\R_{>0};\C^n)}^2\right)<\infty.
 \end{align*}
Here we used that $L_\rho^2(\R_{>0};H)\hookrightarrow L_{\rho_0}^2(\R_{>0};H)$ as contraction.
Next, let $\varphi,\psi\in C_c^\infty(\R;\R^n)$. Then we obtain
\begin{align*}
\int_{\mathbb{R}}|F(\varphi)(t)-F(\psi)(t)|^{2}\e^{-2\rho t}\text{ d}t & =\int_{\mathbb{R}}|\tilde{f}(t,\varphi(t)+y_{0})-\tilde{f}(t,\psi(t)+y_{0})|^{2}\e^{-2\rho t}\text{ d}t\\
 & =\int_{\mathbb{R}_{>0}}|f(t,\varphi(t)+y_{0})-f(t,\psi(t)+y_{0})|^{2}\e^{-2\rho t}\text{ d}t\\
 &\leq \int_{\mathbb{R}_{>0}}c^{2}(|\varphi(t)-\psi(t)|)^{2}\e^{-2\rho t}\text{ d}t\leq c^{2}\|\varphi-\psi\|_{L_{\rho}^{2}}^{2}.
\end{align*}
Since $F(0)\in L^2_\rho(\R;\C^n)$, the shown estimate yields $F(\varphi)\in L^2_\rho(\R;\C^n)$ for each $\varphi\in C_c^\infty(\R;\C^n)$ as well as the eventual (0,0)-Lipschitz continuity of $F$.
\end{proof}

The next result is concerned with the well-posedness for Caputo fractional
differential equations. We shall use the characterization of the Caputo
differential equation outlined in Theorem \ref{thm:Int_Cap}.

\begin{thm}
\label{thm:wpCap}Let $y_{0}\in\mathbb{C}^{n}$. Then there is $\rho_{1}>0$
such that for all $\rho\geq\rho_{1}$ there exists a unique $y\in L_{\rho}^{2}(\mathbb{R};\mathbb{C}^{n})$
with $y-y_{0}\chi_{\R_{> 0}}\in H_{\rho}^{\alpha}(\mathbb{R};\mathbb{C}^{n})$
satisfying
\[
\partial_{0,\rho}^{\alpha}(y-y_{0}\chi_{\R_{> 0}})=\tilde{f}(\cdot,y(\cdot)).
\]
Moreover, $\spt y\subseteq\mathbb{R}_{\geq0}$.
\end{thm}

\begin{proof}
With $F$ as defined in Proposition \ref{prop:Lips}, we may apply
Corollary \ref{cor:sol_theory_Lipschitz} with $\beta=\gamma=0$ to
obtain unique existence of $z\in H_{\rho}^{\alpha}(\mathbb{R};\mathbb{C}^{n})$
such that
\[
\partial_{0,\rho}^{\alpha}z=F_{\rho}(z).
\]
Setting $y\coloneqq z+y_{0}\chi_{\R_{> 0}}$, we obtain in turn unique existence
of a solution of the desired equation. Since $\spt F_{\rho}(z)\subseteq\mathbb{R}_{\geq0}$,
we obtain with Proposition \ref{prop:causality_fractional} that $\spt z=\spt \partial_{0,\rho}^{-\alpha}F_{\rho}(z)\subseteq\mathbb{R}_{\geq0}$.
Thus, $\spt y\subseteq\mathbb{R}_{\geq0}$.
\end{proof}

We remark here that the condition $\spt y\subseteq\mathbb{R}_{\geq0}$
together with $y-y_{0}\chi_{\R_{> 0}}\in H_{\rho}^{\alpha}(\mathbb{R};\mathbb{C}^{n})$
describes, how the initial value $y_{0}$ is attained. Indeed, if
$\alpha$ is large enough (e.g. $\alpha>1/2$) so that $H_{\rho}^{\alpha}(\mathbb{R};\mathbb{C}^{n})$
is a subset of functions for which the limit at $0$ exists, then
the mentioned conditions imply
\[
0=(y-y_{0}\chi_{\R_{> 0}})(0-)=(y-y_{0}\chi_{\R_{> 0}})(0+)=y(0+)-y_{0},
\]
that is, the initial value is attained.

We conclude this section by having a look at the case of the Riemann\textendash Liouville
fractional differential equation \eqref{eq:RL}.
To this end, we note that $\chi_{\R_{>0}}y_0\in H_\rho^0(\R;H)$ for $\rho>0$ and by Example \ref{exa:disder} we have
\[
\partial_{0,\rho}\chi_{\R_{>0}}y_0 = \delta_0y_0\in H_\rho^{-1}(\R;H).
\]
We also recall the notation $g_\beta(t)\coloneqq \frac{1}{\Gamma(\beta+1)}t^\beta\chi_{\R_{>0}}$ for $\beta,t\in\R$.

\begin{prop}
Let $y_0\in\C^n$.
Assume that $C_c^\infty(\R;H)\ni\varphi\mapsto\tilde f(\cdot,\varphi(\cdot))$ is eventually $(\alpha-1,\alpha-1)$-Lipschitz continuous
and denote with $H_\rho^{\alpha-1}(\R;H)\ni y\mapsto \tilde f_\rho(\cdot,y(\cdot))$ its Lipschitz-continuous extension for some $\rho>0$.
There is $\rho_1>0$ such that for $\rho\geq\rho_1$ we have a unqiue solution $y\in H_\rho^{\alpha-1}(\R;H)$ of the equation
\[
\partial_{0,\rho}^\alpha y=y_0\delta_0+\tilde f_\rho(\cdot,y(\cdot)),\qquad y\in H_\rho^{\alpha-1}(\R;H),
\]
with $\partial_{0,\rho}^{\alpha-1}y-y_0\chi_{\R_{>0}}\in H_\rho^0(\R;H)$ and $\spt(y)\subseteq\R_{\geq 0}$.
\end{prop}

\begin{proof}
The mapping $G$ defined by
\[
G(\varphi)(t)\coloneqq \tilde f(t,\partial_{0,\rho}^{1-\alpha}\varphi(t)+g_{\alpha-1}(t)y_0), \qquad \varphi\in C_c^\infty(\R;H),\;t\in\R
\]
is eventually $(0,\alpha-1)$-Lipschitz continuous.
Indeed, this fact follows from $g_{\alpha-1}y_0\in H_\rho^{\alpha-1}(\R;H)$ and the unitarity of $\partial_{0,\rho}^{\alpha-1}:H_\rho^{\alpha-1}(\R;H)\to H_\rho^0(\R;H)$.
Let $\rho_1> 0$ be such that $\tilde f_\rho$ and therefore $G_\rho$ exist for $\rho\geq\rho_1$.
Let $\rho\geq\rho_1$.
The Riemann\textendash Liouville equation is equivalent to
\[
\partial_{0,\rho}^{\alpha-1}y - \chi_{\R_{>0}}y_0 = \partial_{0,\rho}^{-1}\tilde f_\rho(\cdot, y(\cdot)),\qquad y\in H_\rho^\alpha(\R;H).
\]
With the transformation $z=\partial_{0,\rho}^{\alpha-1}y-\chi_{\R_{>0}}y_0$
and using $\partial_{0,\rho}^{1-\alpha}\chi_{\R_{>0}}y_0=\partial_{0,\rho}g_{\alpha}y_0=g_{\alpha-1}y_0$ (cf.\ Corollary \ref{c:poly})
this equation is equivalent to
\[
\partial_{0,\rho}z=G_\rho(z),\qquad z\in H_\rho^0(\R;H).
\]
By Corollary \ref{cor:sol_theory_Lipschitz} (with $\gamma=\alpha-1$) we find a unique solution $z\in H_\rho^0(\R;H)$.
We have $\spt (\partial_{0,\rho}^{-1}G_\rho(\cdot))\subseteq \R_{\geq 0}$.
By Proposition \ref{prop:causality_fractional} $\spt z\subseteq \R_{\geq 0}$.
Hence we have a unqiue solution $y=\partial_{0,\rho}^{1-\alpha}(z-\chi_{\R_{>0}}y_0)\in H_\rho^{\alpha-1}(\R;H)$ of the Riemann-Liouville equation
with $\spt y \subseteq \R_{\geq 0}$ and $\partial_{0,\rho}^{\alpha-1}y-\chi_{\R_{>0}}y_0=z\in H_\rho^0(\R;H)$.
\end{proof}

\begin{rem}
The space $H_\rho^0(\R;H)$ is continuously embedded into $H_\rho^{\alpha-1}(\R;H)$.
Thus, the assumption that $C_c^\infty(\R;H)\ni\varphi\mapsto\tilde f(\cdot,\varphi(\cdot))$ is eventually $(\alpha-1,\alpha-1)$-Lipschitz continuous,
can be replaced by the stronger assumption that $C_c^\infty(\R;H)\ni\varphi\mapsto\tilde f(\cdot,\varphi(\cdot))$ is eventually $(\alpha-1,0)$-Lipschitz continuous,
which might be easier to compute.
\end{rem}

\bibliography{PDE}
\bibliographystyle{plain}

\end{document}